%% file: Bilinear_Kakeya_in_H_1.tex
\documentclass[11pt,letterpaper]{amsart}
\usepackage{amsmath, amssymb, amsthm}
\usepackage{geometry}
\usepackage{hyperref}
\usepackage{enumitem}
\usepackage{biblatex}
\input{preamble}

\title{\textbf{A Bilinear Kakeya Inequality in the Heisenberg Group}}
\author{Yannis Galanos}

\address{Yannis Galanos, School of Mathematics and Maxwell Institute for Mathematical Sciences, University of Edinburgh, James Clerk Maxwell Building, Peter Guthrie Tait Road, King's Buildings, Edinburgh EH9 3FD, United Kingdom}
\email{I.Galanos@sms.ed.ac.uk}
\date{}

\newtheorem{theorem}{Theorem}[section]
\newtheorem{lemma}[theorem]{Lemma}
\newtheorem{remark}[theorem]{Remark}
\newtheorem{definition}[theorem]{Definition}
\bibliography{references}

\numberwithin{equation}{section}
\numberwithin{proposition}{section}

\begin{document}
\maketitle

\begin{abstract}

We prove a bilinear Kakeya inequality in the first Heisenberg group and a sharp bilinear Kakeya estimate for Euclidean curved tubes in $\R^2$. By adapting an argument of F\"assler, Pinamonti and Wald \cite{FPW} involving Heisenberg projections, we show that the latter implies the former. We prove the estimate for curved tubes using a combination of techniques developed by Pramanik, Yang and Zahl \cite{PYZ}, Wolff \cite{Wolff} and Schlag \cite{Schlag}. We introduce a novel broadness hypothesis inspired by Zahl \cite{Zahl2}, which rules out bush-type configurations that break transversal structure. We argue that such a hypothesis is needed for proving the bilinear estimates we present. We also introduce necessary additional linear terms to the estimate to counteract Szemer\'edi--Trotter-type clustering phenomena.
\end{abstract}

\section{Introduction and Main Results}
\subsection{An Overview of the Problem}\label{sec:heisenberg}
The study of the Kakeya problem in the Heisenberg group $\bbH$ has recently gained traction due to numerous works that reveal connections with Euclidean curved Kakeya estimates, planar incidence geometry and restricted projection theory
\cite{FOP},\cite{FP},\cite{FPW},\cite{Liu}. Multilinear Kakeya and restriction estimates are powerful tools in Euclidean harmonic analysis, notably in the works of Tao, Vargas, Vega \cite{TVV}, Bennett, Carbery, Tao \cite{BCT} and Bourgain, Guth \cite{BG}. In \cite{FPW}, F\"assler, Pinamonti and Wald proved a maximal Kakeya inequality for metric tubes in $\mathbb{H}$, using Heisenberg projections to reduce the problem for Heisenberg tubes to a maximal function estimate for thickened parabolas proven by Pramanik, Yang and Zahl \cite{PYZ}. The latter is a generalisation of the circular maximal function estimate proven by Wolff \cite{Wolff}. The proofs of these planar maximal function estimates rely on a ``bilinear-to-linear" argument. It is therefore natural to ask whether bilinear analogues of those theorems exist. We propose such an (almost) bilinear inequality and prove it by reducing it to a tangency-counting estimate for parabolic arcs,  adapting techniques from \cite{PYZ} and \cite{Zahl2}. We also introduce a quantitative broadness hypothesis inspired by \cite{Zahl2}, which rules out ``narrow" configurations that are not compatible with bilinear structure. Using techniques similar to \cite{FPW}, we show that the planar bilinear estimate implies a Kakeya estimate for Heisenberg tubes, establishing a bilinear analogue of the maximal estimate in \cite{FPW}. By adapting classical geometric constructions (e.g. ``bushes" of tubes and ``clamshell-like" configurations of curves), we show that a broadness condition such as the one we introduce here (Definitions \ref{def:trans}, \ref{def:broad}) is needed in order to preserve non-trivial bilinear structure in this setting. Lastly, we show that the $L^p$ exponents and power of $\delta$ in the planar Kakeya estimate we obtain are sharp and discuss further issues concerning sharpness.
\subsection{Kakeya in the Heisenberg Group}
The Heisenberg group $\bbH$ is the prototypical example of a nilpotent Lie group (see \cite{Capogna} for more context). It is defined as $\mathbb{R}^3$ with the following group operation:
\[
(x, y, t) \cdot (x', y', t') = \left(x + x', y + y', t + t' + \frac{1}{2}(xy' - x'y)\right).
\]
The Kor\'anyi gauge $\|x\|_{\mathbb{H}} := \big((x^2 + y^2)^2 + 16t^2\big)^{1/4}$ defines a distance in $\bbH$, which induces a left-invariant metric on $\bbH$,
$$d(x,y):=\|y^{-1}\cdot x\|_{\mathbb{H}},$$
called the \textit{Kor\'anyi} metric. The topology induced on $\bbH$ by $d$ is the Euclidean topology. As a locally compact group, $\bbH$ admits a left Haar measure, which (up to multiplicative constants) is identified with the Lebesgue measure on $\R^3$. While the topological dimension of $\bbH$ is $3$, the \textit{homogeneous} dimension (i.e. the Hausdorff dimension with respect to the Kor\'anyi metric) is $4$.  The Heisenberg group also admits a homogeneous dilation:
$$\delta_\lambda(x,y,t):=(\lambda x,\lambda y, \lambda^2t),$$
with which it becomes a homogeneous Lie group. The Heisenberg Lie algebra $\mathfrak{h}$ is generated by the two left-invariant vector fields
$$X=\partial_x-\frac{y}{2}\partial_t,\ \ \ \ \ \ Y=\partial_y+\frac{x}{2}\partial_t,$$
and their commutator $T:=[X,Y]=\partial_t$. A useful concept in the study of $\bbH$ is that of a \textit{horizontal} curve. A smooth curve $\gamma:[0,1]\to\bbH$ is said to be horizontal if there exist smooth scalar functions $a,b$ such that
\begin{equation*}
    \dot{\gamma}(t)= a(t)X(\gamma(t))+b(t)Y(\gamma(t)).
\end{equation*}
It is known that a smooth curve in $\bbH$ has finite length (i.e. finite $1$-dimensional Hausdorff measure $\cH_d$) if and only if it is horizontal \cite{Capogna}. In the context of the Heisenberg Kakeya problem, we are only interested in directions that produce curves of finite length, as the implicit goal is to show that Kakeya sets are ``small". For this reason, we limit our considerations to horizontal lines. These are exhausted by left-translations of lines $\ell_{a,b}=\{(sa,sb,0): s\in\R\}$ in the $xy$-plane (hence the name horizontal). Motivated by the above reasoning, we define Heisenberg tubes as Korányi $\delta$-neighbourhoods of horizontal lines.
\begin{definition}
    Let  $\delta>0,\ e \in \mathbb{S}^1 \times \{0\}\equiv\bbS^1$ and $x\in \R^3$. We define the Heisenberg $\delta$-tube with horizontal direction $e$ centred at $x$ as 
    $$T_e^\delta(x)=\{x\cdot se\cdot z:\  s\in[-1/2,1/2],\ z\in B_d(0,\delta)\}$$
    Equivalently, we may define $T_e^\delta(x)$ as the Korányi $\delta$-neighbourhood of the horizontal line segment $\ell_{x,e}:=\{x\cdot se: s\in[-1/2,1/2]\}$, which we call the \textit{core} of $T_e^\delta(x)$. We will sometimes denote $T_e^\delta(x)$ by $T^\delta_{\ell_{x,e}}$. Notation will always be clear from context.
\end{definition}
 F\"assler, Pinamonti and Wald \cite{FPW} proved the following Kakeya inequality for Heisenberg tubes, which they also proved is equivalent to the sharp $L^3$ estimate for the Heisenberg Kakeya maximal function.
 \begin{theorem}[F\"assler--Pinamonti--Wald \cite{FPW}]\label{thm:FPW}
      Let $\varepsilon>0$. For every $\delta\in(0,1)$ and every collection $\mathcal{T}$ of Heisenberg $\delta$-tubes with $\delta^2$-separated directions in $\bbS^1$ we have
\begin{equation}\label{eq:kakeya}
    \int\limits_{\R^3}\Big( \sum_{T \in \mathcal{T}} \chi_T\Big)^{3/2}\lesssim_\varepsilon \delta^{3 - \varepsilon} \#\mathcal{T}.
\end{equation}
 \end{theorem}
The proof of Theorem \ref{thm:FPW} is based on the fact that Heisenberg $\delta$-tubes map into $\delta^2$-neighbourhoods of parabolas via Heisenberg projections. This reduces (\ref{eq:kakeya}) to a Kakeya estimate for thickened parabolic arcs in the plane. It is also shown by the authors of \cite{FPW} that the latter is a special case of Proposition $2.1$ in \cite{PYZ}:
\begin{theorem}[Pramanik--Yang--Zahl \cite{PYZ}]\label{thm:pyz}
    Let $\varepsilon>0$. Let $I$ be a compact interval and let $\gamma: I\to \bbS^2$ be a $C^2$ curve satisfying $\rm{span}\{\gamma(s),\gamma'(s),\gamma''(s)\}=\R^3$, for all $s\in I$. Then there exists $\delta_0$ depending only on $\varepsilon,\gamma$ such that the following holds for all $\delta\in(0,\delta_0]:$
    Let $Z_\delta\subseteq B(0,1)\subseteq\R^3$ be a $\delta$-separated set satisfying a Frostman-type non-concentration condition: 
    $$\#\big(Z_\delta\cap B\big)\leq\delta^{-\varepsilon}(r/\delta), \mbox{ for all balls}\ B\subseteq\R^3 \ \mbox{ of radius }\ r\geq\delta.$$
    Then
    \begin{gather}
        \int_{[0,1]^2}\Big(\sum_{z\in Z_\delta} \chi_{\Gamma_z^\delta}\Big)^{3/2}\leq \delta^{1-C\varepsilon}\# Z_\delta,
    \end{gather}
    where $C>0$ is an absolute constant and $\Gamma_z^\delta$ is the Euclidean $\delta$-neighbourhood of the graph of $\Gamma_z=\{(s,\langle\gamma(s), z\rangle):\ s\in I\}.$
\end{theorem}
In similar fashion to \cite{BCT}, one can attempt to formulate a multilinear version of (\ref{eq:kakeya}), where one considers multiple so-called \textit{transversal} collections of tubes. The notion of transversality for Heisenberg tubes is slightly different from the one in \cite{BCT}. Here, the set of allowed (horizontal) directions is identified with $\bbS^1$, which means that there can be at most two linearly independent directions. Therefore, in order to have a meaningful notion of transversality, we may only consider bilinear estimates. In contrast to \cite{BCT}, where the (bilinear) version of Theorem $1.15$ (in \cite{BCT}) in $\R^2$ is trivial (the numerology gives optimal $L^p$ exponent equal to $1$), the natural exponent for the bilinear theorem in the Heisenberg group is $(3/2)/2=3/4$, making such a theorem non-trivial.\\\\
The value of such a bilinear estimate is not limited to the role of a ``stepping stone" towards the analogous (already established in \cite{FPW}) linear theorem. Multilinear inequalities in general have had profound implications in harmonic analysis, beyond the scope of facilitating progress towards proving their linear counterparts. This has been achieved through techniques such as the so-called \textit{broad-narrow} method utilised in \cite{BG} to prove a trilinear curved Kakeya estimate in $\R^4$. Below in this subsection, we argue that such ``broadness" considerations will be crucial in determining the correct formulation of the main bilinear estimates in this paper and subsequently proving them. We also remark that such obstacles are not present in works that deal with linear estimates \cite{FPW},\cite{Wolff},\cite{Zahl}. The reason for this fundamental difference between the present work and previous ones will be explained in Subsection \ref{subsec:broadness}.
\begin{definition}\label{def:trans}
    Let $c\geq0$ be small constant, let $e_1,e_2$ be the standard basis vectors in $\R^3$ parallel to the $x$ and $y$ axes respectively and let $\delta\in(0,1)$. We say that the pair of families of Heisenberg $\delta$-tubes $(\cT_1,\cT_2)$ is \textit{transversal} if for any $T_j\in\mathcal{T}_j$, $j=1,2$, the horizontal direction $e(T_j)$ of $T_j$ satisfies $|e(T_j)-e_j|\leq c$.
\end{definition}
 A naive bilinear conjecture motivated by (\ref{eq:kakeya}) and \cite{BCT} is the following. 
 \begin{conjecture}\label{conj:naive}
     Let $\varepsilon>0$. For every $\delta\in(0,1)$ and every pair of transversal families of Heisenberg $\delta$-tubes  $(\cT_1,\cT_2)$ we have 
 \end{conjecture}
\begin{equation}\label{eq:naive}
    \int\limits_{\R^3}\Big(\sum\limits_{T_1\in\cT_1}\chi_{T_1}\sum\limits_{T_2\in\cT_2}\chi_{T_2}\Big)^{3/4}\lesssim\delta^{4-\varepsilon}\#\cT_1^{3/4}\#\cT_2^{3/4}.
\end{equation}
To convince the reader that this is a sensible proto-conjecture, we remark that the powers of $\delta$ (up to $\varepsilon$-loss) and $\#\cT_1,\#\cT_2$ on the right-hand side of \eqref{eq:naive} are the best one can hope to obtain, given the left-hand side of \eqref{eq:naive}. Indeed, taking $\cT_1$ and $\cT_2$ to be singletons shows that the power of $\delta$ is maximal (since the intersection of two transversal tubes is $\sim\delta^4$, while considering repeated tubes in $\cT_1$ and $\cT_2$ shows that the powers of $\#\cT_1,\#\cT_2$ must be at least as large as the exponent of the integrand of the left-hand side. Even so, it was observed by Katrin F\"assler (personal communication) that (\ref{eq:naive}) is false, even when $c=0$. The following example is inspired by F\"assler's observation:
\begin{example}\label{ex:bush}  
Fix $\delta\in(0,1)$. Let $\Omega\subset\bbS^1$ be an arc of Euclidean length $\delta^{3/2}$ centred at $e_1$ and let $\Omega_\delta$ be a maximal $\delta^2$-separated subset of $\Omega$ that includes $e_1$. For each $\omega\in\Omega_\delta$ we define $T_{1,\omega}=T_{\omega}^\delta(0)$. Let $I=\{se_1: s\in[-\delta^{3/4},\delta^{3/4}]\}$ and $P$ be a maximal $2\delta$-separated subset of $I$. For each $p\in P$ let $T_{2,p}=T_{e_2}^\delta(p)$. The families of Heisenberg $\delta$-tubes $\cT_1=\{T_{1,\omega}\}_{\omega\in\Omega_\delta}$ and $\cT_2=\{T_{2,p}\}_{p\in P}$ satisfy the following properties:
\begin{itemize}
    \item The tubes $T_{2,p}$ are pairwise disjoint and each one intersects every $T_{1,\omega}$ inside a region of volume $\sim\delta^4$,
    \item $\#\cT_2=\# P\sim\delta^{-1/4}$.
\end{itemize}
Thus for small $\varepsilon>0$,
$$\int\limits_{\R^3}\Big(\sum\limits_{T_1\in\cT_1}\chi_{T_1}\sum\limits_{T_2\in\cT_2}\chi_{T_2}\Big)^{3/4}\sim\delta^4\#\cT_1^{3/4}\#\cT_2\gg \delta^{4-\varepsilon}\#\cT_1^{3/4}\#\cT_2^{3/4}.$$
\end{example}
 By considering the symmetric example where the roles of $\cT_1$ and $\cT_2$ are swapped, we may also conclude that the left-hand side of (\ref{eq:naive}) can be $\gtrsim\delta^4\#\cT_1\#\cT_2^{3/4}$. Therefore, in order to get an estimate with better exponents on $\#\cT_1,\#\cT_2$, it is necessary to impose some additional hypotheses. Heuristically, the family $\cT_1$ in Example \ref{ex:bush} is a ``narrow bush", in the sense all horizontal directions concentrate in the small arc $\Omega_\delta$. In order to have non-trivial bilinear structure, one needs to impose a ``broadness" hypothesis that forbids too many directions from concentrating in small arcs (in proportion to the arc length) and thus ruling out narrow ``bushes". This intuition leads us to the following definition:
 \begin{definition}[Broadness Condition for Horizontal Directions]\label{def:bush}
     Let $C>0$ be a constant, $\delta,\alpha>0$ and let $\cL$ be a family of horizontal line segments of unit length. For each $\ell\in\cL$, let $e(\ell)\in\bbS^1$ denote the horizontal direction of $\ell$. For $\sigma\in[\delta,1]$, denote the Kor\'anyi $\sigma$-neighbourhood of $\ell\in\cL$ by $T_\ell^\sigma$. We say that $\cL$ is $(\delta,\alpha)$-broad if for every $z\in\bbH^1,\sigma\in[\delta,1]$ and every arc $\Omega\subset\bbS^1$, we have
     \begin{gather}
         \#\{\ell\in\cL:\ T_\ell^\sigma\cap B_d(z,C\sigma)\neq\emptyset,\ e(\ell)\in\Omega\}\lesssim1+|\Omega|^\alpha\#\{\ell\in\cL:\ T_\ell^\sigma\cap B_d(z,C\sigma)\neq\emptyset\}.
     \end{gather}
 \end{definition}
 In addition to bushes, there is another class of examples that we need to take account of. The reader may verify that by only keeping $T_{1,e_1}$ in $\cT_1$ from Example \ref{ex:bush}, the left-hand side of (\ref{eq:naive}) is comparable to $\delta^4\#\cT_2$. This is similar in spirit to the hairbrush example by Wolff \cite{Wolff2}, and does not fall into the category of bush-like examples. Therefore, the right-hand side of any non-trivial estimate must look like  $\delta^4\Big(\#\mathcal{T}_1^{3/4} \#\mathcal{T}_2^{3/4} + \#\cT_1+\#\cT_2\Big)$, which resembles the right-hand side of several combinatorial estimates, such as the Szemer\'edi-Trotter point-line incidence estimate and more importantly in the context of this work, Wolff's bilinear circle-tangency rectangle bound (Lemma $1.4$ in \cite{Wolff}). This leads us to the statement of the main result.
\begin{theorem}\label{thm:main}
Let $\varepsilon,\delta\in(0,1)$ be sufficiently small and let $\alpha\in[\frac{1}{2},1]$. Let $\mathcal{T}_1$, $\mathcal{T}_2$ be finite collections of Heisenberg $\delta$-tubes of which the associated collections of core lines $\cL_1,\cL_2$ are $(\delta,\alpha)$-broad and such that the horizontal direction $e(\ell_j)$ of each $\ell_j\in\mathcal{L}_j$ satisfies $|e(\ell_j) - e_j|\leq c$ for some small absolute constant $c\in(0,\frac{\sqrt{2}-1}{2})$. Then we have
\begin{equation}\label{eq:main}
\int_{\mathbb{R}^3} \Big( \sum_{T_1 \in \mathcal{T}_1} \chi_{T_1} \Big)^{3/4} \Big( \sum_{T_2 \in \mathcal{T}_2} \chi_{T_2} \Big)^{3/4} \lesssim_\varepsilon  \delta^{4 - \varepsilon} \left( \#\mathcal{T}_1^{3/4} \#\mathcal{T}_2^{3/4} + \#\mathcal{T}_1 + \#\mathcal{T}_2 \right).
\end{equation}
\end{theorem}
Following a similar strategy to \cite{FPW}, we will reduce Theorem \ref{thm:main} to a bilinear analogue of Theorem \ref{thm:pyz} which we state in the following subsection (Theorem \ref{thm:kakeya}).
\subsection{The Bilinear Curved Kakeya Estimate in $\R^2$}\label{sec:kakeya}
The following definitions are standard in the literature of incidence geometry (\cite{Wolff},\cite{Zahl}). Let $I$ be a compact interval. For $C^2$ functions $f,g:\R^2\to \R$ we define
\[
\tau(f, g) := \sup_{\theta \in I} \Big[|f(\theta) - g(\theta)| + |f'(\theta) - g'(\theta)| + |f''(\theta) - g''(\theta)|\Big].
\]
In what follows, we will identify every quadratic $\frac{a}{2}s^2+bs+c$ with its corresponding triple of coefficients $(a,b,c)$.
\begin{definition}[Bipartite pair]
Let $\rho>0$. A pair of families $(\mathcal{C}, \mathcal{D})$ of quadratics is called a \emph{$\rho$-bipartite pair} if 
\begin{gather*}
    \rho\leq \tau(f, g) \leq 100\rho, \mbox{ for all } f\in\mathcal{C},g\in\mathcal{D},\\
    \tau(f,g)\leq\rho, \mbox{ if } f,g\in\mathcal{C} \mbox{ or }f,g\in\mathcal{D}.
\end{gather*}
\end{definition}
\begin{definition}[Curvilinear Rectangles]
    Let $0<\delta\leq t\leq 1$. For any interval $J\subseteq I$, let $f^\delta(J)$ (or simply $f^\delta$ if $J=I$) be the Euclidean $\delta$-neighbourhood of the graph of $f$ over $J$. We call $f^\delta(J)$ a curvilinear rectangle. Additionally, if $J$ has length $\sqrt{\frac{\delta}{t}}$, we call $f^\delta(J)$ a $(\delta,t)$-rectangle. The area of a $(\delta,t)$-rectangle is $\sim \delta^{3/2}t^{-1/2}$.
\end{definition}
\begin{definition}[Tangency]\label{def:tangency}
 We fix an absolute constant $C>0$. Let $J\subseteq I$ be an interval and let $R=g^\delta(J)$. We say that $f$ and $R$ are tangent if $R\subseteq f^{C\delta}(J)$ and write $f\sim R$.
\end{definition}
\begin{definition}[Quantitative Broadness]\label{def:broad}
    Let $\delta\in(0,1)$ and $\alpha>0$. Let $\mathcal{Q}\subseteq\R^3$ be a $\delta$-separated family of quadratic polynomials. We say that $\mathcal{Q}$ is $(\delta,\alpha)$-broad if there exists an absolute constant $c_0\in(0,1)$ such that for every $\sigma\in[\delta,1],\ t\in[\sigma,1]$ and every $(\sigma,t)$-rectangle $R$ we have
    \begin{gather}\label{eq:broad}
       \# \{f\in\mathcal{Q}:\ f\sim R\}\lesssim 1+ t^\alpha \#\mathcal{Q}.
    \end{gather}
\end{definition}
\begin{remark}
    Informally, Definition \ref{def:broad} asserts that long rectangles cannot be tangent to too many curves. An assumption of this type is needed in order to rule out ``narrow" examples which are not amenable to bilinear techniques and which indeed violate any non-trivial estimates. We shall present such examples in detail in Section \ref{sec:4}.
\end{remark}
\begin{remark}
    Definition \ref{def:broad} is similar in spirit to Definition $2.4 $ in \cite{Zahl2}. The main differences are
    \begin{enumerate}
        \item The latter is a local property, in the sense that it requires the number of curves tangent to a large rectangle $R$ to be small compared to a \textit{broad} ``base" $(\delta,t)$-rectangle $R_0$. In contrast, Definition \ref{def:broad} only requires the number of curves tangent to a large rectangle $R$ to be small compared to the total number of curves, which is a strictly weaker hypothesis, except in degenerate cases where all curves are tangent to a single rectangle.
        \item In \cite{Zahl2}, Zahl only requires the condition to hold for an exponent $\alpha=\varepsilon$, which can be made arbitrarily small at the cost of a $\delta^{-\varepsilon}$ loss in the right-hand side of the main estimate (Theorem $1.11$) in \cite{Zahl2}. In the main theorem of this paper, we require $\alpha$ to be larger than an absolute number. However, as we remarked above, this hypothesis is not restrictive at all when the total number of curves is $\gtrsim\delta^{-1}$, and in general tends to be more restrictive for families $\cQ$ with small cardinalities.
    \end{enumerate}
\end{remark}
\begin{remark}
     Definition \ref{def:broad} has similar structure to the one of Definition \ref{def:bush}. As we will see in Section \ref{sec:2} (where we project Heisenberg tubes to parabolas), $(\delta,\alpha)$-broadness of tubes implies $(\delta,\alpha)$-broadness of their planar projections. This is a crucial step for reducing Theorem \ref{thm:main} to the following planar curved Kakeya estimate.
\end{remark}
\begin{theorem}\label{thm:kakeya}
    Let $\varepsilon>0$, $0<\delta\leq\rho\leq 1,\ \alpha\geq1/2$ and $(\mathcal{F}, \mathcal{G})$ be a $\rho$-bipartite pair such that $\mathcal{F}$ and $\mathcal{G}$ are each $\delta$-separated in the metric $\tau$ and $\mathcal{F},\mathcal{G}\subseteq B_\tau(0,100)$. Moreover, suppose $\mathcal{F}$ and $\mathcal{G}$ are $(\delta, \alpha)$-broad. Then we have
    \begin{equation}\label{eq:parabola}
    \int_{[0,1]^2} \Big( \sum_{f \in \mathcal{F}} \chi_{f^\delta} \Big)^{3/4} \Big( \sum_{g \in \mathcal{G}} \chi_{g^\delta} \Big)^{3/4} \lesssim_\varepsilon \rho^{-1/2} \delta^{3/2-\varepsilon} \left( \#\mathcal{F}^{3/4} \#\mathcal{G}^{3/4} + \#\mathcal{F} + \#\mathcal{G} \right).
\end{equation}
\end{theorem}
%\begin{remark}
   % The conclusion of Theorem \ref{thm:kakeya} with the exponent $3/4$ replaced by $p\in[3/4,1]$ follows from Theorem \ref{thm:kakeya} by simply interpolating with the trivial estimate for $p=1$.
%\end{remark}
The rest of the paper is divided into three sections: 
\begin{enumerate}
    \item In Section \ref{sec:2} we reduce Theorem \ref{thm:main} to Theorem \ref{thm:kakeya}. We first localise the estimate to a unit Korányi ball. We then show that each Heisenberg $\delta$-tube maps into the $\delta^2$-neighbourhood of a parabolic arc via Heisenberg projection. Lastly, we show that the broadness condition for the Heisenberg tubes implies the broadness condition for the parabolas.
    \item In Section \ref{sec:3} we prove Theorem \ref{thm:kakeya} using tools from \cite{Wolff},\cite{Zahl} to reduce (\ref{eq:parabola}) to an incidence-geometric lemma for curvilinear rectangles which can be found in \cite{Zahl}, but whose original formulation and proof is in \cite{Wolff}. We also make crucial use of the broadness assumption for $\cF$ and $\cG$.
    \item In Section \ref{sec:4} we show that the $L^p$ and $\delta$ exponents in Theorem \ref{thm:kakeya} are sharp and discuss further issues related to sharpness of other exponents.\\\\
\end{enumerate}
We introduce some standard notation that will be used throughout the paper. Let $(X,\sigma)$ be a metric space.
\begin{itemize} 
    \item  We denote the $s$-dimensional Hausdorff measure on $X$ with respect to $\sigma$ by $\mathcal{H}^s_\sigma$.
    \item We denote the open ball centred at $x\in X$ of radius $r>0$ by $B_\sigma(x,r)$. Additionally, we set $B_0:=B_\sigma(0,1)$ (the metric $\sigma$ will always be clear from context).
    \item For a set $A\subseteq\mathbb{R}^d$, denote the Euclidean $\delta$-neighbourhood of $A$ by $N_\delta(A)$.
    \item For a set $S\subseteq\R^3$ and $\mathcal{T}_j,\ j=1,2$ collections of Heisenberg tubes, let $\mathcal{T}_j(S):=\{T_j\in\mathcal{T}_j:\ T_j\cap S\neq\emptyset\},\ j=1,2$.
    \item  We let $C$ denote a generic non-negative absolute constant.
    \item For $A,B>0,\eta\neq1$ we write 
        \begin{itemize}
            \item $A\lesssim B$ if $A\leq CB$. We also write $A\sim B$ if $A\lesssim B\lesssim A$.
            \item $A\lesssim_\eta B$ if there exists a constant $C=C(\eta)$ such that $A\leq C(\eta)B$.
            \item $A\lessapprox_\eta B$ if $A\lesssim|\log \eta|B$ and $A\approx_\eta B$ if $A\lessapprox_\eta B\lessapprox_\eta A$.
        \end{itemize}
    \item If $I$ is an interval centred at $x$ with length $L$ and $C>0$ is a constant, then we denote the interval centred at $x$ with length $L/C$ by $I/C$.
 \end{itemize} 
\section*{Acknowledgements}
This research was inspired by discussions surrounding the multilinear Kakeya problem in non-Euclidean settings. I am grateful to my advisor Anthony Carbery for his guidance, as well as Malabika Pramanik and Joshua Zahl for their valuable input through insightful discussions. I would also like to thank Kaiyi Huang, Betsy Stovall and Sarah Tammen for sharing their results from \cite{Stovall} prior to public dissemination, as well as Katrin F\"assler for sharing her work on this problem with my advisor through personal communication. This research was funded by EPSRC and the Maxwell Institute for Mathematical Sciences.
\section{Proof of the Bilinear Heisenberg Kakeya Estimate}\label{sec:2}
\subsection{Preliminary Reductions and Heisenberg Projections}\label{sec:localise}
We show that (\ref{eq:main}) from the main Theorem \ref{thm:main} can be deduced from the following local version:
\begin{equation}\label{eq:local}
    \int_{B_0} \Big( \sum_{T_1 \in \mathcal{T}_1(B_0)} \chi_{T_1} \Big)^{3/4} \Big( \sum_{T_2 \in \mathcal{T}_2(B_0)} \chi_{T_2} \Big)^{3/4} \lesssim_\epsilon \delta^{4 - \epsilon} \left( \#\mathcal{T}_1(B_0)^{3/4} \#\mathcal{T}_2(B_0)^{3/4} + \#\mathcal{T}_1(B_0) + \#\mathcal{T}_2(B_0) \right).
\end{equation}
Indeed, let $\mathcal{B}$ be a collection of finitely overlapping Korányi balls of radius $1$ such that $\R^3=\bigcup\limits_{B\in\mathcal{B}}B$.  By definition, the left-translate of a Heisenberg $\delta$-tube is also a Heisenberg $\delta$-tube. This is also true for $d$-metric balls. Suppose that $B\in\mathcal{B}$ is centred at $p\in\R^3$. By an elementary calculation, one can see that the map $q\mapsto p\cdot q$ has Jacobian equal to 1. By changing variables  we get, for each $B\in\mathcal{B}$,
\begin{gather*}
    \int_{B} \Big( \sum_{T_1 \in \mathcal{T}_1} \chi_{T_1}(q) \Big)^{3/4} \Big( \sum_{T_2 \in \mathcal{T}_2} \chi_{T_2}(q) \Big)^{3/4}dq=
    \int_{B_0} \Big( \sum_{T_1 \in \mathcal{T}_1} \chi_{T_1}(p\cdot q )\Big)^{3/4} \Big( \sum_{T_2 \in \mathcal{T}_2} \chi_{T_2}(p\cdot q) \Big)^{3/4}dq\\
    =\int_{B_0}\Big( \sum_{T_1' \in \mathcal{T}_1'} \chi_{T_1'} \Big)^{3/4} \Big( \sum_{T_2' \in \mathcal{T}_2'} \chi_{T_2'} \Big)^{3/4} 
\end{gather*}
where $\mathcal{T}_j'=p^{-1}\mathcal{T}_j,\ j=1,2$. Suppose we know (\ref{eq:local}). Then, the last display is
\begin{gather*}
        \lesssim_\epsilon  \delta^{4 - \epsilon} \left( \#\mathcal{T}_1'(B_0)^{3/4} \#\mathcal{T}_2'(B_0)^{3/4} + \#\mathcal{T}_1'(B_0) + \#\mathcal{T}_2'(B_0) \right)\\
         =\delta^{4 - \epsilon} \left( \#\mathcal{T}_1(B)^{3/4} \#\mathcal{T}_2(B)^{3/4} + \#\mathcal{T}_1(B) + \#\mathcal{T}_2(B) \right).
\end{gather*} 
We now sum over $B\in\mathcal{B} $ to get
\begin{gather*}
    \int_{\mathbb{R}^3} \Big( \sum_{T_1 \in \mathcal{T}_1} \chi_{T_1} \Big)^{3/4} \Big( \sum_{T_2 \in \mathcal{T}_2} \chi_{T_2} \Big)^{3/4} \lesssim\sum_{B\in\mathcal{B}} \int_{B} \Big( \sum_{T_1 \in \mathcal{T}_1} \chi_{T_1} \Big)^{3/4} \Big( \sum_{T_2 \in \mathcal{T}_2} \chi_{T_2} \Big)^{3/4} \\
    \lesssim \delta^{4 - \epsilon}\sum_{B\in\mathcal{B}} \left( \#\mathcal{T}_1(B)^{3/4} \#\mathcal{T}_2(B)^{3/4} + \#\mathcal{T}_1(B) + \#\mathcal{T}_2(B) \right)\\
    \sim\delta^{4 - \epsilon}\Bigg[\sum_{B\in\mathcal{B}} \left( \#\mathcal{T}_1(B)^{3/4} \#\mathcal{T}_2(B)^{3/4}\right) + \#\mathcal{T}_1 + \#\mathcal{T}_2\Bigg]\\
    \leq  \delta^{4 - \epsilon}\Bigg[\left(\sum_{B\in\mathcal{B}} \#\mathcal{T}_1(B)^{3/2}\right)^{1/2} \left(\sum_{B\in\mathcal{B}}\#\mathcal{T}_2(B)^{3/2}\right)^{1/2} + \#\mathcal{T}_1 + \#\mathcal{T}_2\Bigg]\\
    \leq  \delta^{4 - \epsilon}\Bigg[\left(\sum_{B\in\mathcal{B}} \#\mathcal{T}_1(B)\right)^{3/4} \left(\sum_{B\in\mathcal{B}}\#\mathcal{T}_2(B)\right)^{3/4} + \#\mathcal{T}_1 + \#\mathcal{T}_2\Bigg]\\
    \sim \delta^{4 - \epsilon}\Bigg[ \#\mathcal{T}_1^{3/4} \#\mathcal{T}_2^{3/4} + \#\mathcal{T}_1 + \#\mathcal{T}_2\Bigg]
\end{gather*}
where in the last two inequalities we used Cauchy-Schwarz and the $\ell^1\hookrightarrow\ell^{3/2}$ embedding.\\\\
In what follows, we fix $I=[-10,10]$. We begin by reducing (\ref{eq:local}) to an inequality for parabolas in the plane. We take an approach similar to Section $2.2$ in \cite{FPW}. Let $\mathfrak{v}=(1,1,0)$ and define $$\mathbb{W}:=\rm{span}\{\mathfrak{v},e_3\}=\{(x,x,t):\ x,t\in\mathbb{R}\},\ \mathbb{L}:=\mathbb{W}^\perp=\{(s,-s,0):\ s\in\mathbb{R}\}.$$ 
For purposes of notational simplicity, we will identify each $(x,x,t)\in\mathbb{W}$ with the point $(x,t)\in\R^2$. Consider the Heisenberg projection onto $\mathbb{W}$ given by $$\pi_\mathbb{W}(x,y,t)=\Big(\frac{x+y}{2},t+\frac{1}{4}(x^2-y^2)\Big).$$ 
The following lemma is essentially Lemma $2.18$ in \cite{FPW}. It states that a Heisenberg $\delta$-tube projects via $\pi_\mathbb{W}$ into the $\delta^2$-neighbourhood of a parabolic arc  -- as long as its direction is far from $\bbL$.
\begin{lemma}\label{le:proj}
    Let $\ p=(x,y,t)\in B_d(0,1),\ e=(a,b,0)\in\mathbb{S}^1\times\{0\}$ such that $\dist(e,\bbL)\geq 1/2$. Let $T=T_e^\delta(p)\cap B_d(0,1)$ and define $\gamma_{T}(s):=\pi_\mathbb{W}(p\cdot se)$. Then there exists an absolute constant $c>0$ such that
    $$\pi_{\mathbb{W}}(T)\subseteq N_{c\delta^2}\big(\Gamma_{T}(I)\big),$$
    where $\Gamma_{T}(I)$ is the graph of $\gamma_T$ restricted on $I$.
\end{lemma}
\begin{proof}
    Let $q\in T$. By definition, $q=p\cdot se\cdot z$  for some $s\in\R,\ z=(z_1,z_2,z_3)\in B_d(0,\delta)$. By left-invariance and the triangle inequality we have 
    $$|s|=\|se\|_\bbH=d(p\cdot se,p)\leq  d(p\cdot se, q)+d(q,0)+ d(0,p)=\|z\|_\bbH+\|q\|_\bbH +\|p\|_\bbH\leq 3.$$
    We have 
    $$\gamma_T(s)=\Bigg(\frac{x+y}{2}+\frac{a+b}{2}s,\frac{a^2-b^2}{4}s^2+\frac{1}{2}(a+b)(x-y)s+t+\frac{1}{4}(x^2-y^2)\Bigg).$$
    Since $\dist(e,\bbL)\geq 1/2$, we have $|a+b|\geq 1/2$. Therefore, by setting $\theta_0(p)=\frac{x+y}{2}$ we may reparametrise by $\theta(s)=\frac{a+b}{2}s+\theta_0(p)$ to get 
    \begin{gather}
        f_T(\theta(s)):=\gamma_T(s)=\Big(\theta,\ \kappa(e)(\theta-\theta_0(p))^2+m(p)(\theta-\theta_0(p))+v(p)\Big),
    \end{gather}
    where $\kappa(e):=\frac{a-b}{a+b},\ m(p):=x-y,\ v(p):=t+\frac{1}{4}(x^2-y^2)$. This reparametrisation will be useful later. Moreover,
    $$\pi_\mathbb{W}(q)=\gamma_T(\theta)+\Big(\frac{z_1+z_2}{2},\ (a-b)(z_1+z_2)\frac{s}{2}+\frac{1}{2}(x-y)(z_1+z_2)+\frac{1}{4}(z_1^2-z_2^2)+z_3\Big).$$
    Since $\dist(e,\mathbb{L})\geq c$, we have $a+b\geq c$ and thus
    $$\Bigg|\pi_\bbW(q)-\gamma_T(s+\frac{z_1+z_2}{a+b})\Bigg|=\Bigg|\pi_\mathbb{W}(q)-f_T(\theta+\frac{z_1+z_2}{2})\Bigg|=\Bigg|\frac{1}{4}(z_1^2-z_2^2)+\frac{1}{4}\frac{a^2-b^2}{(a+b)^2}(z_1+z_2)^2+z_3\Bigg|\lesssim\delta^2.$$
    We conclude the proof by showing that $s+\frac{z_1+z_2}{a+b},\theta+\frac{z_1+z_2}{2}\in I$. Indeed, since $s\in[-3,3]$, we have  $\theta-\theta_0(p)\in\Big[-3{(a+b)}/{2},3({a+b})/{2}\Big]$. Moreover, $\big|\frac{a+b}{2}\big|\leq1$ and $|\theta_0(p)|\leq\|p\|_\bbH\leq1$, so that $\theta\in[-4,4]$. We also have $\big|\frac{z_1+z_2}{2}\big|\leq\|z\|_\bbH\leq1$. Thus, $\theta+\frac{z_1+z_2}{2}\in[-5,5]$ and $s+\frac{z_1+z_2}{a+b}\in[-10,10]$. Notice that we also have
    \begin{equation*}
        \pi_\bbW(T)\subseteq N_{c\delta^2}(Gr(f_T)),
    \end{equation*}
    where $Gr(f_T)$ is the graph of $f_T$ over $I$.
    \end{proof}   
The next lemma is similar to Lemma 2.10 in \cite{FPW}:
\begin{lemma}\label{le:haus}
 Let $\delta\in(0,1), w\in\mathbb{W}, e\in\mathbb{S}^1\times\{0\}$ satisfying $|\mathfrak{v}-e|\geq1/2$, $p=(\xi,t)\in\R^2\times \R$ and $T=T_{e}^\delta(p)$ be the Heisenberg $\delta$-tube centred at $p$ with direction $e$. Then we have $$\mathcal{H}_d^1\big(T\cap\pi_\mathbb{W}^{-1}(\{w\})\big)\lesssim\delta.$$ Moreover, if $w\notin\pi_\mathbb{W}(T)$ then $\mathcal{H}_d^1\big(T\cap\pi_\mathbb{W}^{-1}(\{w\})\big)=0$. We may combine the two conclusions into one as follows:
 \begin{equation}\label{eq:lemma2}
     \mathcal{H}_d^1\big(T\cap\pi_\mathbb{W}^{-1}(\{w\})\big)\lesssim\delta \chi_{\pi_\bbW (T)}(w).
 \end{equation}
\end{lemma}
\begin{proof}
   The second claim follows trivially from the first, since if $w\notin\pi_\bbW(T),$ then $T\cap\pi_\mathbb{W}^{-1}(\{w\})=\emptyset$. We now prove the first claim. Let $w=(\zeta,t')\in\pi_\bbW(T)$. Since $d$ is left-invariant, it follows immediately that $\mathcal{H}_d^1$ is also left-invariant. Therefore, \begin{gather*}
       \mathcal{H}_d^1\big(T\cap\pi_\mathbb{W}^{-1}(\{w\})\big)=\mathcal{H}_d^1\big(w^{-1}T\cap\mathbb{L}\big)=\mathcal{H}_d^1\big(T_e^\delta(w^{-1}p)\cap\mathbb{L}\big).
    \end{gather*}
    Since the map $s\mapsto(s,-s,0)$ is a quasi-isometry from $(\R,|\cdot|)$ to $(\R^3,d)$, we have
$$\mathcal{H}_d^1\big(T_e^\delta(w^{-1}p)\cap\mathbb{L}\big)\sim\mathcal{H}_{|\cdot|}^1(\{s\in\R:\ (s,-s,0)\in T_e^\delta(w^{-1}p)\}.$$
       Moreover, the map $P:(\mathbb{R}^3,d)\to(\mathbb{R}^2,|\cdot|)$ given by $ P(x,y,t)=(x,y)$ is $1$-Lipschitz which implies that $$P(T_e^\delta(w^{-1}p))\subseteq S:=[\xi-\zeta+\rm{span}\{P(e)\}]^\delta.$$
Thus,
    \begin{gather*}
       \mathcal{H}_{|\cdot|}^1(\{s\in\R:\ (s,-s,0)\in T_e^\delta(w^{-1}p)\}\leq\mathcal{H}_{|\cdot|}^1(\{s\in\mathbb{R}:\ (s,-s)\in S\}).
    \end{gather*}
        Since $|\mathfrak{v}-e|\geq1/2$, the line $\{(s,-s):\ s\in\R\}$ intersects $S$ roughly along its short direction. Therefore, $$\mathcal{H}_{|\cdot|}^1(\{s\in\mathbb{R}: (s,-s)\in S\})\lesssim\delta.$$
\end{proof}
\subsection{Reducing Theorem \ref{thm:main} to Theorem \ref{thm:kakeya}}
For each $T_j=T_{p(T_j)}^\delta(e)\in\mathcal{T}_j$, Lemma \ref{le:proj} implies that $\pi_\mathbb{W}(T_j\cap B_0)$ is contained in $\Gamma_{T_j}^{\delta^2}:=N_{c\delta^2}\big(\Gamma_{T_j}(I)\big)$. As one can see from the proof of Lemma \eqref{le:proj}, we also have \begin{equation*}
    \Gamma_{T_j}^{c\delta^2}\subseteq E_{T_j}^{c\delta^2}:= N_{c\delta^2}(Gr(f_{T_j})).
\end{equation*}  
For reasons that will become apparent shortly, we need the following lemma.
\begin{lemma}\label{le:bipartite}
    We divide $B_\tau(0,100)$ into $O(1)$ finitely overlapping $1$-balls. For each $j=1,2$ there exists a subfamily $\cT_j'\subseteq\cT_j$ and a popular $1$-ball $B_j$ such that \begin{gather}\label{eq:pigeon}
    \int\limits_{B_0}\Big(\sum_{T_1\in\mathcal{T}_1(B_0)}\chi_{T_1}\Big)^{3/4}\Big(\sum_{T_2\in\mathcal{T}_2(B_0)}\chi_{T_2}\Big)^{3/4}\sim\int\limits_{B_0}\Big(\sum_{T_1\in\mathcal{T}_1'}\chi_{T_1}\Big)^{3/4}\Big(\sum_{T_2\in\mathcal{T}_2'}\chi_{T_2}\Big)^{3/4}
    \end{gather}
and
    \begin{gather*}
        (\kappa(e(T)),m(p(T)),v(p(T)))\in B_j, \mbox{ for all } T\in\cT_j',\ j=1,2.
    \end{gather*}
    Here, $f_T,f_T',\kappa,m,v$ are as in Lemma \eqref{le:proj}.
\end{lemma}
\begin{proof}
    We split the tubes $T_j\in\cT_j$ into subfamilies depending on which 1-ball contains the point $(\kappa(e(T_j)),m(p(T_j)),v(p(T_j)))$, and write the integral on the left-hand side of \eqref{eq:pigeon} as a sum over all all contributions coming from these $O(1)$ subfamilies. By dyadic pigeonholing, there exist $\cT_1',\cT_2'$ whose corresponding contribution is at least an $O(1)$ proportion of the left-hand side of \eqref{eq:pigeon}.
\end{proof}
By Lemma \ref{le:bipartite} and Cauchy-Schwarz, we have 
\begin{gather*}
    \int\limits_{B_0}\Big(\sum_{T_1\in\mathcal{T}_1(B_0)}\chi_{T_1}\Big)^{3/4}\Big(\sum_{T_2\in\mathcal{T}_2(B_0)}\chi_{T_2}\Big)^{3/4}\sim\int\limits_{B_0}\Big(\sum_{T_1\in\mathcal{T}_1'}\chi_{T_1}\Big)^{3/4}\Big(\sum_{T_2\in\mathcal{T}_2'}\chi_{T_2}\Big)^{3/4}\\
    =\int\limits_{\mathbb{W}}\int\limits_{\mathbb{L}}\Big(\sum_{T_1\in\mathcal{T}_1'}\chi_{T_1}(w\cdot l)\Big)^{3/4}\Big(\sum_{T_2\in\mathcal{T}_2'}\chi_{T_2}(w\cdot l)\Big)^{3/4}\chi_{B_0}(w\cdot l)\ dl\ dw\\
    \leq\int\limits_{\mathbb{W}\cap B_0}\Bigg(\int\limits_{\mathbb{L}\cap B_0}\Big(\sum_{T_1\in\mathcal{T}_1'}\chi_{T_1}(w\cdot l)\Big)^{3/2}dl\Bigg)^{1/2}\Bigg(\int\limits_{\mathbb{L}\cap B_0}\Big(\sum_{T_2\in\mathcal{T}_2'}\chi_{T_2}(w\cdot l)\Big)^{3/2}dl\Bigg)^{1/2}dw.
\end{gather*}
For each $w\in\mathbb{W},\ j=1,2$ we have 
$$\int\limits_{\mathbb{L}}\Big(\sum_{T_j\in\mathcal{T}_j'}\chi_{T_j}(w\cdot l)\Big)^{3/2}dl\leq\Big(\sum_{T_j\in\mathcal{T}_j'}\chi_{\Gamma_{T_j}^{\delta^2}}(w)\Big)^{1/2}\sum_{T_j\in\mathcal{T}_j'}\mathcal{H}_d^1\big(T_j\cap\pi_\mathbb{W}^{-1}(w)\big)\lesssim\delta\Big(\sum_{T_j\in\mathcal{T}_j'}\chi_{\Gamma_{T_j}^{\delta^2}}(w)\Big)^{3/2},$$
where in the first inequality we used the fact that $\chi_{T_j}(w\cdot l)\leq\chi_{\pi_{\bbW}(T_j\cap B_0)}(w)\leq \chi_{\Gamma_{T_j}^{\delta^2}}(w)$ for $w\in\bbW,\ l\in\bbL$, and in the second one we used (\ref{eq:lemma2}) from Lemma \ref{le:haus}. Hence, we obtain 
\begin{gather*}
    \int\limits_{B_0}\Big(\sum_{T_1\in\mathcal{T}_1'}\chi_{T_1}\Big)^{3/4}\Big(\sum_{T_2\in\mathcal{T}_2'}\chi_{T_2}\Big)^{3/4}\leq\delta\int\limits_{\mathbb{W}\cap B_0}\Big(\sum_{T_1\in\mathcal{T}_1'}\chi_{\Gamma_{T_1}^{\delta^2}}(w)\Big)^{3/4}\Big(\sum_{T_2\in\mathcal{T}_2'}\chi_{\Gamma_{T_2}^{\delta^2}}(w)\Big)^{3/4}dw\\
    \leq\delta\int\limits_{\mathbb{W}\cap B_0}\Big(\sum_{T_1\in\mathcal{T}_1'}\chi_{E_{T_1}^{\delta^2}}(w)\Big)^{3/4}\Big(\sum_{T_2\in\mathcal{T}_2'}\chi_{E_{T_2}^{\delta^2}}(w)\Big)^{3/4}dw
\end{gather*} and so (\ref{eq:local}) is reduced to proving
\begin{gather}\label{eq:reduced}
    \int\limits_{\mathbb{W}\cap B_0}\Big(\sum_{T_1\in\mathcal{T}_1'}\chi_{E_{T_1}^{\delta^2}}(w)\Big)^{3/4}\Big(\sum_{T_2\in\mathcal{T}_2'}\chi_{E_{T_2}^{\delta^2}}(w)\Big)^{3/4}dw\\
    \lesssim_\epsilon\delta^{3-\epsilon}\Big[\#\mathcal{T}_1'^{3/4}\#\mathcal{T}_2'^{3/4}+\#\mathcal{T}_1'+\#\mathcal{T}_2'\Big].
\end{gather}
Next, we show that transversality of $\mathcal{T}_1'$ and $\mathcal{T}_2'$ implies that $\{f_{T_1}: T_1\in\cT_1\},\{f_{T_2}:\ T_2\in\cT_2\}$ form a $1$-bipartite pair. 
\begin{lemma}
    Let $\cT_1,\cT_2$ be as in the statement of Theorem \ref{thm:main}. Then there exist $\cT_1'\subseteq\cT_1(B_0),\cT_2'\subseteq\cT_2(B_0)$ such that $\#\cT_1'\sim\#\cT_1(B_0),\ \#\cT_2'\sim\#\cT_2(B_0)$ and for every $T,T'\in\cT_1'\cup\cT_2'$, we have 
    \begin{gather*}
        \tau(f_{T},f_{T'})\leq 1, \mbox{ if } T,T' \mbox{ belong to the same family},\\
        1\leq \tau(f_{T},f_{T'})\leq 100, \mbox{ if } T,T' \mbox{ belong to different families},
    \end{gather*}
\end{lemma}
\begin{proof}
By construction, we have
\begin{gather*}
    \tau(f_T,f_{T'})\leq1, \mbox{ for all } T,T'\in\cT_j',\ j=1,2.
\end{gather*}
Note that if $h(s)=as^2+bs+c,$ then
\begin{equation*}
    \tau(h,0)\lesssim|I|^2(|a|+|b|+|c|)+|I|(2|a|+|b|)+2|a|\lesssim|a|+|b|+|c|.
\end{equation*}
Moreover, since $h''$ is constant, we have $\tau(h,0)\geq|h''|=2|a|$. Recall from Lemma \ref{le:proj} that 
\begin{equation*}
    f_T(\theta)=\kappa(e(T))(\theta-\theta_0(p(T)))^2+m(p(T))(\theta-\theta_0(p(T)))+v(T)
\end{equation*}
and the maps $\kappa,v,m$  are locally Lipschitz with respect to the Euclidean metric. For every $T\in\cT_1', T'\in\cT_2'$, we have
\begin{gather*}
    \tau(f_T,f_T')\lesssim|\kappa(e(T))-\kappa(e(T'))|+|m(p(T))-m(p(T'))|+|v(p(T))-v(p(T'))|\leq100.
\end{gather*}
and $|e(T)-e_1|,|e(T')-e_2|\leq c$. since $\kappa(e_1)=1,\kappa(e_2)=-1$, by choosing $c$ sufficiently small we can ensure that $\kappa(e(T))\in[1/2,3/2]$ and
$\kappa(e(T'))\in[-3/2,-1/2]$. Therefore
\begin{gather*}
    \tau(f_T,f_{T'})\geq|f_T''-f_{T'}''|=2|\kappa(e(T))-\kappa(e(T'))|\geq 1.
\end{gather*}
\end{proof}
To close this section, we show that the non-concentration condition for the Heisenberg tubes implies the non-concentration condition for their projections. This will complete the reduction of Theorem \ref{thm:main} to Theorem \ref{thm:kakeya}.
\begin{lemma}\label{le:jet}
    Let $0<\delta\leq t\leq1,\ \theta\in[0,1]$ and let $J$ be the interval of length $\sqrt{\frac{\delta}{t}}$ centred at $\theta$. Let $R=g^\delta(J)$ be a $(\delta,t)$-rectangle and $f$ be a quadratic polynomial. Then $f$ is tangent to $R$ if and only if  
    \begin{gather}\label{eq:jet}
        |f(\theta)-g(\theta)|\lesssim\delta,\ |f'(\theta)-g'(\theta)|\lesssim\sqrt{\delta t},\ |f''(\theta)-g''(\theta)|\lesssim t.
    \end{gather}
\end{lemma}
\begin{proof}
    Let $h=f-g$. Suppose that $f\sim R$. By definition of tangency, $g^\delta(J)\subseteq f^{C\delta}(J)$ and thus
    $$|h(s)|\leq2\delta,\ \mbox{ for all}\ s\in J.$$
    Thus, $|h(\theta)|\lesssim\delta$ is evidently true. Since $h$ is a quadratic polynomial, we can write
    $$h(s)=h(\theta)+h'(\theta)(s-\theta)+\frac{h''(\theta)}{2}(s-\theta)^2$$
    $$|f(\theta)-g(\theta)|\leq\delta,\ |f'(\theta)-g'(\theta)|\leq\sqrt{\delta t},\ |f''(\theta)-g''(\theta)|\leq t.$$
    Conversely, if \eqref{eq:jet} holds, then by the Taylor expansion of $h$ around $\theta$ we immediately get  $|h(s)|\lesssim\delta$ for all $s\in J$, and thus $R=h^\delta(J)\subseteq f^{C\delta}(J)$
\end{proof}
\begin{proposition}
    Let $C>0$ be an absolute constant and let $\delta,\alpha>0$. Let $\cL$ be a family of unit horizontal line segments that is $(\delta,\alpha)$-broad and such that $\ell\cap B_d(0,C)\neq\emptyset$ for all $\ell\in\cL$. Let $\mathcal{Q}:=\{\pi_\bbW(\ell):\ \ell\in\cL\}$. Then $\mathcal{Q}$ is $(\delta^2,\alpha)$-broad.
\end{proposition}
\begin{proof}
    For each $\ell\in\cL$ let $f_\ell$ be the quadratic such that $\pi_\bbW(\ell)=Gr(f_\ell)$. Let $\sigma\in[\delta^2,1],t\in[\delta^2,1]$ and let $R=g^{\sigma}(J)$ be a $(\sigma,t)$-rectangle. We want to show that 
    \begin{gather}
        \#\{\ell\in\cL:\ f_\ell\sim R\}\lesssim t^\alpha\#\cL.
    \end{gather}
    Let $\theta^*$ be the midpoint of $J$ and $$v^*:=g(\theta^*),\ m^*:=g'(\theta^*),\ \kappa^*:=g''.$$ Let $\ell\in\cL$ be such that $f_\ell\sim R$ and let $T=T_\ell^{\sqrt{\sigma}}=T_{e(T)}^{\sqrt{\sigma}}(p(T))$ be the $\sqrt{\sigma}$-neighbourhood of $\ell$. Note that $f_\ell$ and $f_T$ from Lemma \ref{le:proj} coincide (up to translation of the argument). Thus, $f_T\sim R$ and Lemma \ref{le:jet} yields
    \begin{gather}\label{eq:jet1}
        |f_T(\theta^*)-v^*|\leq\sigma,\ |{f_T}'(\theta^*)-m^*|\leq \sqrt{ \sigma t},\ |{f_T}''-\kappa^*|\leq t.
    \end{gather}
    We will show that there exists an absolute constant $C>0$ and a point $p^*\in B_d(0,C)$ such that for every $\ell\in\cL$ with $f_\ell\sim R$, $T^{\sqrt{\sigma}}_\ell$ intersects $B_d(p^*,C\sqrt{\sigma})$. We first identify $\bbW $ with $\R^2$ via $w_{X,Y}\equiv (X,Y)$. Let $p^*=w_{\theta^*,v^*}\cdot(-m^* e) $, where $e=(1,-1,0)$. We have $\|p^*\|_\bbH\lesssim1$. Let $s^*$ be such that $\theta^*=\theta(s^*)$ (using the parametrisation from the statement of Lemma \ref{le:proj}). Fix $\ell\in\cL$ and let $T=T^\sigma_\ell=T^\sigma_{e(T)}(p(T))$. We claim that
    $$p(T)\cdot s^* e(T)=w_{\theta^*,f_T(\theta^*)}\cdot (f_T'(\theta^*)\ e). $$
    Indeed, by definition of $\gamma_T$ and $f_T$ (in Lemma \ref{le:proj}) we have
    $$\pi_\bbW (p(T)\cdot s^*e(T))=\gamma_T(s^*)=(\theta^*,f_T(\theta^*))\equiv w_{\theta^*,f_T(\theta^*)}, $$
    Since $\pi_\bbW$ is smooth with bounded (above and below) Jacobian, it is locally Lipschitz with respect to the Euclidean metric. Thus, by the triangle inequality and (\ref{eq:jet1}) we get
    \begin{gather*}
        d(p(T)\cdot s^* e(T),p^*)=d(w_{\theta^*,f_T(\theta^*)}\cdot(f_T'(\theta^*)\ e), w_{\theta^*,v^*}\cdot(-m^* e))\\
        \leq d(w_{\theta^*,f_T(\theta^*)}\cdot(f_T'(\theta^*)\ e), w_{\theta^*,v^*}\cdot(f_T'(\theta^*)\ e))+d(w_{\theta^*,v^*}\cdot(f_T'(\theta^*)\ e), w_{\theta^*,v^*}\cdot(-m^* e))\\
        \leq |w_{\theta^*f_T(\theta^*)}\cdot(f_T'(\theta^*)\ e)- w_{\theta^*,v^*}\cdot(f_T'(\theta^*)\ e)|^{1/2}+|f_T'(\theta^*)-m^*|\\
        \leq |\pi_\bbW(w_{\theta^*,f_T(\theta^*)}\cdot(f_T'(\theta^*)\ e))-\pi_\bbW(w_{\theta^*,v^*}\cdot(f_T'(\theta^*)\ e))|^{1/2}+|f_T'(\theta^*)-m^*|\\
        =|w_{\theta^*,f_T(\theta^*)}-w_{\theta^*,v^*} |^{1/2}+|f_T'(\theta^*)-m^*|\\
        =|f_T(\theta^*)-v^*|^{1/2}+|f_T'(\theta^*)-m^*|\lesssim\sigma+\sqrt{\sigma t}\lesssim\sqrt{\sigma},
    \end{gather*}
    so that $p(T)\cdot s^* e(T)\in B_d(p^*,C\sqrt{\sigma})$. We simply observe that $p(T)\cdot s^* e(T)$ lies on the core of $T$, which is $\ell$ by definition. Therefore, $\ell\cap B_d(p^*,C\sqrt{\sigma})\neq\emptyset$ for every $\ell\in\cL$.\\\\
    Since $|f_\ell''-\kappa^*|\leq t$ for all $\ell\in\cL$ and the map $\kappa(a,b)=\frac{a-b}{a+b}$ is bi-Lipschitz, there exists an arc $\Omega\subset\bbS^1$ of length $\sim t$ such that  $e_\ell\in\Omega$ for all $\ell\in\cL$. Explicitly, $\Omega:=\{e\in\bbS^1:\ |\kappa(e)-\kappa^*|\leq t\}$. Therefore,
    $$\{\ell\in\cL:\ f_\ell\sim R\}\subseteq\{\ell\in\cL:\ e_\ell\in\Omega\}.$$
    Since $\sqrt{\sigma}\in[\delta,1]$, we may apply the $(\delta,\alpha)$-broadness of $\cL$ to $B(p^*,C\sqrt{\sigma})$ and $\Omega$ to get
    $$\#\{\ell\in\cL:\ \ell\cap B_d(p^*,C\sqrt{\sigma})\neq\emptyset,\  e(T)\in\Omega\}\lesssim t^{\alpha}\#\cL.$$
\end{proof}
\section{A Bilinear Kakeya Estimate for Parabolas}\label{sec:3}
The goal in this section is to prove Theorem \ref{thm:kakeya}. We do that by reducing \eqref{eq:parabola} to an incidence estimate for curvilinear rectangles and applying a rectangle-counting lemma from \cite{Zahl}. The proof involves numerous dyadic pigeonholing steps, whereby we iteratively choose a quantitative feature of $(\cF,\cG)$ (e.g. how many pairs $(f,g)$ meet at each point) and break down the left-hand side of \eqref{eq:parabola} into contributions where the size of the chosen feature is roughly constant. We then find a ``typical" size of that feature, the contribution of which is a $\log (\delta^{-1})$ proportion of the left-hand side of \eqref{eq:parabola}. The first dyadic decomposition is applied directly to the left-hand side of \eqref{eq:parabola}. Each subsequent decomposition is applied on the dominant contribution of the previous step. In this way, we retain the constant property of each feature in every subsequent step. Many arguments in this section are inspired by \cite{PYZ}.
\subsection{Tangencies of Planar Curves and Incidence Counting for Curvilinear Rectangles}\label{sec:stage2}
We are reduced to proving a Kakeya-type estimate for $\delta$-neighbourhoods of parabolas.  In what follows, $f$ and $g$ will be quadratic polynomials defined over the interval $I=[-5,5]$. We define the tangency gauge
\[
\Delta(f, g) := \inf_{\theta \in I}\Big[ |f(\theta) - g(\theta)| + |f'(\theta) - g'(\theta)|\Big].
\]
The tangency gauge $\Delta$ first appeared in \cite{Wolff}. It measures the size of the intersection between the $\delta$-neighbourhoods of two parabolic arcs. Note that $\Delta(f, g)$ is a pseudo-metric,
since it is symmetric and satisfies the triangle inequality.  As an immediate consequence of the definition of $\Delta$, we have
\begin{equation*}
    \Delta(f,g)\leq \tau(f,g).
\end{equation*}
The following is a special case of Lemma 3.18 in \cite{PYZ}:
\begin{lemma}\label{le: rect}
    The set $A_\delta:=\{s\in I/4:\ |f(s)-g(s)|\leq\delta\}$ is either empty or the union of at most two intervals of length $O\big(\delta/\sqrt{(\tau(f,g)+\delta)(\Delta(f,g)+\delta)}\Big)$. Moreover, if there exists $\theta\in I/4$ such that $|f(\theta)-g(\theta)|\leq\delta/2$ and $J$ is the interval constituting $A_\delta$ that contains $\theta$, then $|J|\gtrsim\delta/\sqrt{(\tau(f,g)+\delta)(\Delta(f,g)+\delta)}$ and if $a,b$ are the endpoints of $J$, then $\dist(\theta,\{a,b\})\gtrsim\delta/\sqrt{(\tau(f,g)+\delta)(\Delta(f,g)+\delta)}$. 
\end{lemma}
\begin{definition}[Tangency set]
Let $R$ be a curvilinear rectangle. We define the following sets of tangencies associated to $R$:
\begin{gather*}
    \mathcal{F}(R):=\{f\in\mathcal{F}:\ f\sim R\},\\
    \mathcal{G}(R):=\{g\in\mathcal{G}:\ g\sim R\},\\
    \mathbf{T}_{\mathcal{F},\mathcal{G}}(R):=\mathcal{F}(R)\times\mathcal{G}(R).
\end{gather*}
and write $(f,g)\sim R$ whenever $(f,g)\in\mathbf{T}_{\mathcal{F},\mathcal{G}}(R).$
\end{definition}
The following lemma is an a simple consequence of the definitions.
\begin{lemma}
    Let $\delta>0, t\in[\delta,1]$ and let $R=h^\delta(J)$ be a $(\delta,t)$-rectangle. Suppose $f$ and $g$ are tangent to $R$. Suppose there exists an $\theta\in I$ such that $|f(\theta)-g(\theta)|\leq\delta$. Then we have
    \begin{equation*}
    \tau(f,g)\lesssim t.
\end{equation*}
\end{lemma}
\begin{proof}
Let $h:=f-g$ and $t:=\frac{(\tau(f,g)+\delta)(\Delta(f,g)+\delta)}{\delta}$. By Lemma \ref{le: rect}, the set $f^\delta\cap g^\delta$ is a $(\delta,t)$-rectangle $R=f^\delta(J)$ ($J $ is the interval of length $\sqrt{\delta/t}$ centred at $\theta$). By Lemma \ref{le:jet}, $$|h''|=|f''-g''|\lesssim t,\ |h'(s)|\lesssim\sqrt{\delta t},\ |h(s)|\lesssim \delta\ \mbox{ for all }\ s\in J,$$ and thus,
$$|h'(s)|=\Big|\int_{\theta}^sh''(u)du+h'(\theta)\Big|\leq|h''||s-\theta|+|h'(\theta)|\lesssim t|J|+\delta\lesssim t\sqrt{\frac{\delta}{t}}+\delta\lesssim t,\ \mbox{ for all }\ s\in I.$$
Similarly, 
$$|h(s)|=\Big|\int_{\theta}^sh'(u)du+h(\theta)\Big|\lesssim \|h'||_\infty|I|+|h(\theta)|\lesssim t+\delta\lesssim t,\ \mbox{ for all}\ s\in I.$$
Therefore, $\tau(f,g)\lesssim t.$
\end{proof}

    For $p\in [0,1]^2$, let
    $$\mathcal{F}(p):=\{f\in\mathcal{F}:\ p\in f^\delta\},\ \mathcal{G}(p):=\{g\in\mathcal{G}:\ p\in g^\delta\}.  $$
    and define the multiplicity functions
    $$m_\mathcal{F}(p):=\sum_{f\in\mathcal{F}}\chi_{f^\delta}(p),\ m_\mathcal{G}(p):=\sum_{g\in\mathcal{G}}\chi_{g^\delta}(p)$$
    We first show that $m_\mathcal{F}(p),m_\mathcal{G}(p)\lesssim\delta^{-2}$ for all $p\in [0,1]^2.$
    Let $p=(s_0,y_0)\in [0,1]^2$ and let $(a,b,c)$ correspond to some $f\in\mathcal{F}(p)$. Since the graph of $f$ intersects $B(p,\delta)$, we have $|f(s_0)-y_0|\leq\delta$. By replacing $s_0$ with $s_0+\delta/2$ if necessary, we may assume that $|s_0|\geq\delta/2$. Thus, $({a},{b},{c})$, lies in the $C\delta$-neighbourhood of the hyperplane $xs_0^2+ys_0+z=y_0$ intersected with the unit ball in $\R^3$. The same holds for every other triple $(a',b',c')$ that corresponds to some curve in $\cF(p)$. Since such triples are also $\delta$-separated, there can be at most $\sim\delta^{-2}$ quadratics in $\mathcal{F}(p)$. Similarly, $m_\mathcal{G}(p)\lesssim\delta^{-2}$.\\\\
      By dyadic pigeonholing, there exist $\lambda_1,\lambda_2\leq\delta^{-2}$ and a compact $E_1\subseteq[0,1]^2$ such that $m_\cF(p)\sim\lambda_1,\ m_\cG(p)\sim\lambda_2$ for all $p\in E_1$ and
      $$\int_{[0,1]^2} \Big( \sum_{f \in \mathcal{F}} \chi_{f^\delta} \Big)^{3/4} \Big( \sum_{g \in \mathcal{G}} \chi_{g^\delta} \Big)^{3/4}\lessapprox_\delta\int_{E_1} \Big( \sum_{f \in \mathcal{F}} \chi_{f^\delta} \Big)^{3/4} \Big( \sum_{g \in \mathcal{G}} \chi_{g^\delta} \Big)^{3/4}\sim\ \lambda_1^{3/4}\lambda_2^{3/4}|E_1|. $$
       Hence, the proof of (\ref{eq:parabola}) is reduced to
    \begin{equation}\label{eq:incidence}
        \lambda_1^{3/4}\lambda_2^{3/4}|E_1|\lesssim_\varepsilon\rho^{-1/2} \delta^{3/2 - \varepsilon} \left( \#\mathcal{F}^{3/4} \#\mathcal{G}^{3/4} + \#\mathcal{F} + \#\mathcal{G} \right).
    \end{equation}
      We now proceed to set up the proof of (\ref{eq:incidence}). The following lemma allows us to find a common tangency scale $\Delta $ for a large subset of pairs $(f,g)\in\cF\times\cG$. It follows a two-ends reduction scheme similar to Section 5.1.2 in \cite{PYZ}.
      \begin{lemma}\label{le:common}
          There exists a subset $E_3\subseteq E_1$ such that $|E_3|\approx|E_1|$ and $\Delta\in[\delta,1]$ with the following property: for every $p\in E_3$ there exists a pair $(f_p,g_p)\in\cF(p)\times\cG(p)$ such that $\Delta(f_p,g_p)\leq\Delta$ and
          \begin{gather*}
              \#\{f\in \cF(p):\ \Delta(f,g_p)\leq\Delta\}\gtrsim\delta^\varepsilon\lambda_1,\\
               \#\{g\in \cG(p):\ \Delta(f_p,g)\leq\Delta\}\gtrsim\delta^\varepsilon\lambda_2.
          \end{gather*}
      \end{lemma}
      As a consequence, for every $p\in E_3$ we have $\Delta(f,g)\lesssim\Delta$ for a $\delta^{2\varepsilon}$-fraction of $\cF(p)\times\cG(p)$.
      \begin{proof}
          For $p\in E_1$ and $\Delta>0$ define $$n_1(\Delta,p):=\max_{f\in \cF(p)}\#\{g\in\cG(p):\ \Delta(f,g)\leq\Delta\},$$
          and $$n_2(p)=\sup_{\Delta\in[\delta,100\rho]} n_1(\Delta,p)\Delta^{-\varepsilon}.$$
          For each $p\in E_1$ there exists $\Delta_{\cG,p}$ such that 
          \begin{gather*}
              n_2(p)\leq2 n_1(\Delta_{\cG,p},p)\Delta_{\cG,p}^{-\varepsilon}
          \end{gather*}
          and an $f_p\in\cF(p)$ such that $$n_1(\Delta_{\cG,p},p)=\#\{g\in \cG(p):\ \Delta(f_p,g)\leq \Delta_{\cG,p}\}.$$
          By dyadic pigeonholing on $\Delta_{\cG,p}$ we may find a $E_2\subseteq E_1$ and  $\Delta_{\cG}$ such that $|E_2|\approx_\delta|E_1|$ and $\Delta_{\cG,p}\sim\Delta_\cG$ for all $p\in E_2$. Therefore, we have 
          \begin{gather*}
              \#\{g\in \cG(p):\ \Delta(f_p,g)\leq\Delta_\cG\}\gtrsim\Delta_\cG^\varepsilon n_1(100,p)\gtrsim\Delta_\cG^{\varepsilon}\#\{g\in\cG(p):\ \Delta(f_p,g)\leq100\}\sim\Delta_\cG^{\varepsilon}\lambda_2,
          \end{gather*}
          where the last approximate equality holds because for every $f\in\cF,g\in\cG$ we have $\Delta(f,g)\leq\tau(f,g)\leq 100\rho$. For each $p\in E_2$, let $\tilde{\cG}(p):=\{g\in\cG(p):\ \Delta(f_p,g)\leq\Delta_\cG\}.$ For $p\in E_2$ and $\Delta>0$ we define
          $$m_1(\Delta,p):=\max_{g\in\tilde{\cG}(p)}\#\{f\in\cF(p):\ \Delta(f,g)\leq\Delta\},$$
          and 
          $$m_2(p):=\sup_{\Delta\in[\delta,100]}m_1(\Delta,p)\Delta^{-\varepsilon}. $$
          As before, for each $p\in E_2$ we may find $\Delta_{\cF,p}$ and $g_p\in \tilde{\cG}(p)$ such that 
          $$m_2(p)\leq2m_1(\Delta_{\cF,p},p)\Delta_{\cF,p}^{-\varepsilon}$$ and then perform another dyadic pigeonholing to find $E_3\subseteq E_2$ and $\Delta_\cF$ such that $|E_3|\approx_\delta|E_2|$ and $\Delta_{\cF,p}\sim \Delta_\cF$ for all $p\in E_3$. Consequently,
          \begin{gather*}
              \#\{f\in\cF(p):\ \Delta(f,g_p)\leq\Delta_\cF\}\gtrsim\Delta_\cF^\varepsilon\lambda_1.
          \end{gather*}
          Moreover, since $g_p\in\tilde{\cG}(p),$ we have $\Delta(f_p,g_p)\leq\Delta_\cG$. To conclude the proof, we set $\Delta=\max\{\Delta_\cF,\Delta_\cG\}$ and observe that 
          \begin{gather*}
              \#\{f\in\cF(p):\ \Delta(f,g_p)\leq\Delta_\cF\}\leq\#\{f\in\cF(p):\ \Delta(f,g_p)\leq\Delta\},\\
              \#\{g\in\cG(p):\ \Delta(f_p,g)\leq\Delta_\cG\}\leq\#\{g\in\cG(p):\ \Delta(f_p,g)\leq\Delta\},
          \end{gather*}
          for all $p\in E_3$. The consequence in the statement of the lemma follows from the triangle inequality for $\Delta(f,g)$: For every $(f,g)\in\cF(p)\times\cG(p)$ such that $\Delta(f,g_p)\leq\Delta$ and $\Delta(f_p,g)\leq\Delta$ we have
          \begin{gather*}
              \Delta(f,g)\leq\Delta(f,g_p)+\Delta(f_p,g_p)+\Delta(f_p,g)\leq 3\Delta.
          \end{gather*}
      \end{proof}
We will need the following definitions:
\begin{definition}[Comparable/incomparable rectangles]
Two $(\delta,t)$-rectangles $R,R'$ are said to be \emph{comparable} if there is another $(C\delta,t)$-rectangle $R''$ that contains both. We say that $R,R'$ are \emph{incomparable} if they are not comparable.
\end{definition}
\begin{definition}[Richness]
A curvilinear rectangle ${R}$ is called $(\mu, \nu)$-rich if $\mu\leq\#\cF(R)\leq 2\mu$ and $\nu\leq\#\cG(R)\leq 2\nu$.
\end{definition}
\begin{lemma}\label{le:tangency}
    Let $p=(s,y)\in E_3$ and let \begin{gather*}
    \cF'(p):=\{f\in\cF(p):\ \Delta(f,g_p)\leq\Delta\},\\
    \cG'(p):=\{g\in\cG(p):\ \Delta(f_p,g)\leq\Delta\}. 
\end{gather*} be the sets obtained in Lemma \ref{le:common}. Let $C>0$ be an absolute constant and $J$ be the interval centred at $s$ with length $C\delta/\sqrt{\Delta\rho}$ and define $R_p:=f_p^\delta(J)$. Then every pair $(f,g)\in\cF'(p)\times\cG'(p)$ is tangent to $R_p$.
\end{lemma}
\begin{proof}
    Let $f\in\cF'(p),\ g\in\cG'(p)$. By definition, $\Delta(f_p,g)\leq\Delta$. Moreover, by the triangle inequality we get 
    $$\Delta(f_p,f)\leq\Delta(f_p,g_p)+\Delta(g_p,f)\leq2\Delta.$$ Let $h\in\{f,g\}$. We have $$|h(s)-f_p(s)|\leq|h(s)-y|+|f_p(s)-y|\leq2\delta.$$
    Lemma \ref{le: rect} at scale $4\delta$ implies that there exists an interval $\tilde{J}$ that contains $s$, is contained in the set 
    $$A:=\{t\in I/4: |f_p(t)-h(t)|\leq4\delta\},$$ has length $$\sim \delta/\sqrt{(\tau(f_p,h)+\delta)(\Delta(f_p,h)+\delta)},$$
    and if $\{a,b\}$ are the endpoints of $\tilde{J}$, then $\rm{dist}(s,\{a,b\})\gtrsim\delta/\sqrt{\rho\Delta}$. Since $\Delta(f_p,h)\leq2\Delta$ and $\tau(f_p,h)\leq100\rho$, the length of $\tilde{J}$ is $\gtrsim\delta/\sqrt{\rho\Delta}$. Thus, by choosing $C$ appropriately we ensure that $J\subseteq \tilde{J}$. Consequently, we have $R_p\subseteq h^{C\delta}(J)$, i.e. $h\sim R_p$.\\\\
\end{proof}
    According to Lemma \ref{le:tangency}, for every $p\in E_3$, $R_p$ is tangent to at least $\delta^{\varepsilon}\lambda_1$ curves $f\in\cF(p)$ and $\delta^\varepsilon\lambda_2$ curves $g\in \cG(p)$. In other words, each $R_p$ is $(\delta^\varepsilon\lambda_1,\delta^\varepsilon\lambda_2)$-rich. Since $E_3$ is compact, there exists a finite cover $\mathcal{R}_0$ of $E_3$ consisting of these rectangles $R_p$. By a standard covering lemma, there exists a cover $\mathcal{R}$ of $E_3$ consisting of pairwise incomparable rectangles such that each $R\in\mathcal{R}$ is the $10$-fold dilation of some $R_0\in\mathcal{R}_0$ by a (common) fixed constant. Thus, we have $$|E_3|\lesssim\#\mathcal{R}\cdot\delta^2\Delta^{-1/2}\rho^{-1/2},$$ so that (\ref{eq:incidence}) can be reduced to the following lemma:
\begin{lemma}\label{le:aux}
        Let $0<\delta\leq\rho<1$, $\Delta\in[\delta,100\rho]$ and $(\mathcal{F},\mathcal{G})$ be a $\rho$-bipartite pair as defined above. Then, for any family $\mathcal{R}$ of pairwise incomparable $(\delta^\varepsilon\lambda_1,\delta^\varepsilon\lambda_2)$-rich curvilinear rectangles of dimensions $\delta\times\frac{\delta}{\sqrt{\rho\Delta}}$ we have
            \begin{equation}\label{eq:inc2}
                \lambda_1^{3/4}\lambda_2^{3/4}\#\mathcal{R}\lesssim_{\varepsilon}\delta^{-1/2-\varepsilon}\Delta^{1/2}\left( \#\mathcal{F}^{3/4} \#\mathcal{G}^{3/4} + \#\mathcal{F} + \#\mathcal{G} \right).
            \end{equation}
        \end{lemma}
        For the proof of Lemma \ref{le:aux}, we rely on Lemma $5.18$ from \cite{Zahl}, whose original formulation (for circles) and proof can be found as Lemma $1.4$ in \cite{Wolff}:
    \begin{lemma}[Incidence bound for curvilinear rectangles]\label{le:zahl}
    Let $\varepsilon>0,\ 1>t\geq\eta>0$, $\mu,\nu\in\Z_+$, $(\mathcal{C}, \mathcal{D})$ be a $t$-bipartite pair such that $\mathcal{C},\mathcal{D}$ are $\eta$-separated (with respect to $d$) and let $\tilde{\mathcal{R}}$ be a collection of pairwise incomparable $(\mu,\nu)$-rich (with respect to $(\mathcal{C},\mathcal{D})$\ ) $(\eta,t)$-rectangles. Then
    \[
    \#\tilde{\mathcal{R}} \lesssim_\varepsilon \left(\#\mathcal{C} \#\mathcal{D}\right)^\varepsilon \left[ \left(\frac{\#\mathcal{C} \#\mathcal{D}}{\mu \nu} \right)^{3/4} + \frac{\#\mathcal{C}}{\mu} + \frac{\#\mathcal{D}}{\nu} \right].
    \]
    \end{lemma}
    We also need the following lemma which dispenses with the separation hypothesis. We note that this is merely a technicality which is not present in \cite{Wolff},\cite{Zahl}, as the nature of the results proven in these works require \textit{a priori} $\eta$-separation assumptions on $\mathcal{C},\mathcal{D}$.
    \begin{lemma}[Incidence bound for curvilinear rectangles - 2]\label{le:incid}
        Let $\varepsilon>0,\ 1>t\geq\eta>0$, $\mu,\nu\in\Z_+$, $(\mathcal{C}, \mathcal{D})$ be a $t$-bipartite pair and let $\tilde{\mathcal{R}}$ be a collection of pairwise incomparable $(\mu,\nu)$-rich (with respect to $(\mathcal{C},\mathcal{D})$\ ) $(\eta,t)$-rectangles. Then
    \[
    \#\tilde{\mathcal{R}} \lesssim_\varepsilon \left(\#\mathcal{C} \#\mathcal{D}\right)^\varepsilon \left[ \left(\frac{\#\mathcal{C} \#\mathcal{D}}{\mu \nu} \right)^{3/4} + \frac{\#\mathcal{C}}{\mu} + \frac{\#\mathcal{D}}{\nu} \right].
    \]
    \end{lemma}
    \begin{proof}[Proof of Lemma \ref{le:incid}]
        Let $\mathcal{C},\mathcal{D}$ be as in the statement of the lemma. We will find suitable $\eta$-separated subfamilies $\mathcal{C}_\eta,\mathcal{D}_\eta$ such that each $\tilde{R}\in\tilde{\mathcal{R}}$ is $(\tilde{\mu},\tilde{\nu})$-rich with respect to $(\mathcal{C}_\eta,\mathcal{D}_\eta)$, where $\tilde{\mu},\tilde{\nu}$ satisfy
        $$\frac{\#\mathcal{C}_\eta}{\tilde{\mu}}\sim\log\#\mathcal{C}\Big(\frac{\#\mathcal{C}}{\mu}\Big),\ \ \frac{\#\mathcal{D}_\eta}{\tilde{\nu}}\sim\log\#\mathcal{D}\Big(\frac{\#\mathcal{D}}{\nu}\Big).$$
        The conclusion of the lemma will immediately follow by applying Lemma \ref{le:zahl} for $\mathcal{C}_\eta,\mathcal{D}_\eta$. We identify each quadratic with its triplet of coefficients. The metric $\tau$ can be identified with the Euclidean metric in $\R^3$ via this mapping. Under this identification, we may decompose $\mathcal{C}$ and $\mathcal{D}$ as follows:
         \begin{gather*}
         \mathcal{C}=\bigsqcup_{B_\Delta}\mathcal{C}\cap B_\Delta=\bigsqcup_{1\leq A\leq\#\mathcal{C}}\bigsqcup_{B_\Delta:\ \#\mathcal{C}\cap B_\Delta\sim A}\mathcal{C}\cap B_\Delta=:\bigsqcup_{1\leq A\leq\#\mathcal{C}}\mathcal{C}_A,\\
         \mathcal{D}=\bigsqcup_{B_\Delta} \mathcal{D}\cap B_\Delta=\bigsqcup_{1\leq B\leq\#\mathcal{D}}\bigsqcup_{B_\Delta:\ \#\mathcal{D}\cap B_\Delta\sim B}\mathcal{D}\cap B_\Delta=:\bigsqcup_{1\leq B\leq\#\mathcal{D}}\mathcal{D}_B,
     \end{gather*}
        By dyadic pigeonholing, for each $\tilde{R}\in\tilde{\mathcal{R}},$ there exist popular $A_{\tilde{R}},B_{\tilde{R}}$ such that
        \begin{gather*}
            \mu=\sum_{c\in\mathcal{C},\ c\parallel R }1\approx_{\#\mathcal{C}}\sum_{c\in\mathcal{C}_{A_{\tilde{R}}},\ c\parallel R} 1,\\
            \nu=\sum_{d\in\mathcal{D},\ d\parallel R }1\approx_{\#\mathcal{D}}\sum_{d\in\mathcal{D}_{B_{\tilde{R}}},\ d\parallel R} 1,
        \end{gather*}
     and we have
     $$\tilde{\mathcal{R}}=\bigsqcup_{A,B}\{R\in\tilde{\mathcal{R}}:\ A_{\tilde{R}}=A,\ B_{\tilde{R}}=B\}=:\bigsqcup_{A,B}\tilde{\mathcal{R}}_{A,B}. $$
     By another dyadic pigeonholing, we find $A,B$ such that that $\#\tilde{\mathcal{R}}\lessapprox_{\#\mathcal{C},\#\mathcal{D}}\#\tilde{\mathcal{R}}_{A,B}. $ The family $\tilde{\mathcal{R  }}_{A,B}$ consists of $({\mu},{\nu})$-rich (with respect to $\mathcal{C}_A,\mathcal{D}_B$) $(\eta,t)$-rectangles. Let $\mathcal{C}_\eta, D_\eta$ be maximal $\eta$-separated subfamilies of $\mathcal{C}_A$ and $\mathcal{B}_B$ respectively. By definition, we can write 
     $$\mathcal{C}_A=\bigsqcup_{c\in\mathcal{C}_\eta} B_d(c,\eta),\ \ \mathcal{D}_B=\bigsqcup_{d\in \mathcal{D}_\eta} B_d(d,\eta).$$
     and it follows by definition of $\mathcal{C}_A,\mathcal{D}_B$ that $\#\mathcal{C}_A=A\#\mathcal{C}_\eta$ and $\#\mathcal{D}_B   = B \#\mathcal{D}_\eta$. By pigeonholing, there exist $\tilde{\mu},\tilde{\nu}$ and $\tilde{\mathcal{R}}'\subseteq\tilde{\mathcal{R}}_{A,B}$ such that $\#\tilde{\mathcal{R}}'\approx_{\#\mathcal{C},\#\mathcal{D}}\#\tilde{\mathcal{R}}_{A,B} $ and
     each $\tilde{R}\in\tilde{\mathcal{R}}' $ is $(\tilde{\mu},\tilde{\nu})$-rich with respect to $(\mathcal{C}_\eta , \mathcal{D}_\eta)$. By definition of $\mathcal{C}_A, \mathcal{D}_B$ and $\tilde{\mu},\tilde{\nu}$ we have
     \begin{gather*}
         \mu\approx_{\#\mathcal{C}}\sum_{c\in \mathcal{C}_A,\ c\parallel\tilde{R}}1=\sum_{\tilde{c}\in\mathcal{C}_\eta,\ \tilde{c}\parallel\tilde{R}}\sum_{c\in\mathcal{C}_A\cap B_d(\tilde{c},\eta)}1\sim A\tilde{\mu},\\
         \nu\approx_{\#\mathcal{D}}\sum_{d\in \mathcal{D}_B,\ d\parallel\tilde{R}}1=\sum_{\tilde{d}\in\mathcal{D}_\eta,\ \tilde{d}\parallel\tilde{R}}\sum_{d\in\mathcal{D}_B\cap B_d(\tilde{d},\eta)}1\sim B\tilde{\nu}.
     \end{gather*}
    and thus
    $$\frac{\#\mathcal{C}_\eta}{\tilde{\mu}}\sim\log\#\mathcal{C}\Big(\frac{\#\mathcal{C}}{\mu}\Big),\ \ \frac{\#\mathcal{D}_\eta}{\tilde{\nu}}\sim\log\#\mathcal{D}\Big(\frac{\#\mathcal{D}}{\nu}\Big),$$
    as claimed.
   \end{proof}
    We note that neither Lemma \ref{le:zahl} or Lemma \ref{le:incid} are directly applicable, since the rectangles $R\in\cR$ are not $(\delta,t)$-rectangles for some $t\in[\delta,1]$. We deal with this issue in the following subsection.
   \subsection{Coarse-Scale Rectangles}\label{subsec:coarse}
    We will rescale the rectangles $R\in\cR$ to rectangles $\tilde{R}$ that satisfy the hypotheses of Lemma \ref{le:incid}. We distinguish between the following two cases:\\\\
    \textbf{i) }$\Delta=\delta$. We set $\sigma=\delta$. Each $R\in\mathcal{R}$ is a rectangle of dimensions $\sigma\times\sqrt{\frac{\sigma}{\rho}}$\\
    \textbf{ii) } ${\Delta>\delta}$. We set $\sigma=\Delta$. Each $R\in\mathcal{R}$ is a rectangle of dimensions $\delta\times\frac{\delta}{\sqrt{\rho\sigma}}$.\\\\
    We rescale each $R\in\mathcal{R}$ about its ``central curve" to a $(\sigma,\rho)$-rectangle which we denote by $\Tilde{R}$. To be precise, if $R=f^\delta (J)$ for some quadratic $f$ and an interval $J\subseteq I$, then $\tilde{R}:=f^\sigma (\tilde{J})$, where $\tilde{J}$ is the interval with the same midpoint as $J$ and length $\frac{\sigma}{\delta}|J|=\sqrt{\frac{\sigma}{\rho}}$.\\\\
    We denote the family of coarse rectangles by $\Tilde{\mathcal{R}}$. For each $U\in\tilde{ \mathcal{R}}$, let $$\mathcal{R}(U)=\{R\in\mathcal{R}: \tilde{R}\ \mbox{ comparable to }  U\}. $$
    Since the rectangles $R\in\mathcal{R}$ are pairwise disjoint, comparing dimensions shows that $$\mathcal{R}(U)\leq\frac{\sigma^2}{\delta^2}\leq\delta^{-2},$$ for all $U\in\tilde{\mathcal{R}}$. By a standard covering lemma, we may choose a subfamily $\tilde{\mathcal{R}}_1\subseteq\tilde{\mathcal{R}}$ of pairwise incomparable rectangles such that each  $\tilde{R}\in\tilde{\mathcal{R}}$  is contained in the dilate of some $\tilde{R}_1\in\tilde{\mathcal{R}}_1$ by a (common) fixed constant. Thus, we have
$$\#\mathcal{R}=\sum_{\tilde{R}\in\tilde{\mathcal{R}}_1}\#\mathcal{R}(\tilde{R}).$$
  By dyadic pigeonholing, there exists a constant $M\leq\frac{\sigma^2}{\delta^2}$ and a subfamily $\tilde{\mathcal{R}}_2\subseteq\tilde{\mathcal{R}}_1$ such that $\#\tilde{\mathcal{R}}_2\approx_\delta\#\tilde{\mathcal{R}}_1$ and $\#\mathcal{R}(\tilde{R})\sim M$ for all $\tilde{R}\in\tilde{\mathcal{R}}_2$. Hence,
  $$\#\mathcal{R}\approx_\delta M\#\tilde{\mathcal{R}}_2.$$
We now prove a bound for $M$. For every $\tilde{R}\in\tilde{\mathcal{R}}_2$ and $R\in\mathcal{R}(\tilde{R})$ we have that $f\sim R$ implies $f\sim\tilde{R}$. Indeed, if $R=h^\delta(J), \tilde{R}=h^\sigma(\tilde{J})$, where $J, \tilde{J}$ are concentric intervals of lengths $\frac{\delta}{\sqrt{\rho\sigma}}$ and $\sqrt{\frac{\sigma}{\rho}}$ respectively and $f\sim R$, then by Lemma \ref{le:jet} we have
\begin{gather*}
    |f(\theta)-h(\theta)|\leq\delta\leq\sigma,\ |f'(\theta)-h'(\theta)|\leq\sqrt{\delta\frac{\sigma\rho}{\delta}}=\sqrt{\sigma\rho},
\end{gather*}
and $|f''-h''|\leq\tau(f,h)\lesssim\rho$. Thus, by Lemma \ref{le:jet} $\tilde{R}\sim f$. Moreover, since the rectangles $R\in\mathcal{R}$ are pairwise disjoint, each pair $(f,g)\in\mathcal{F}\times\mathcal{G}$ can only be in $\mathbf{T}_{\mathcal{F},\mathcal{G}}(R)$ for an absolutely bounded number of $R\in\mathcal{R}$. Thus, for every $\tilde{R}\in\tilde{\mathcal{R}}_2$ we have
    \begin{gather*}
        M\lambda_1\lambda_2\sim\sum_{ R\in\mathcal{R}(\tilde{R})}\#\mathbf{T}_{\mathcal{F},\mathcal{G}}(R)= \sum_{R\in\mathcal{R}(\tilde{R})}\sum_{f,g\sim R}1 \\
        \leq\sum_{f,g\sim\tilde{R}}\sum_{R\sim f,g}1 \lesssim\sum_{(f,g)\sim\tilde{R}}1= \#\mathbf{T}_{\mathcal{F},\mathcal{G}}(\tilde{R}).
    \end{gather*}
     By another pigeonholing, there exist integers $\mu,\nu>0$ and $\tilde{\mathcal{R}}_3$ such that $\#\tilde{\mathcal{R}}_3\approx_\delta\#\tilde{\mathcal{R}}_2 $ and each $\tilde{R}\in\tilde{\mathcal{R}}_3 $ is $(\mu,\nu)$-rich with respect to $(\mathcal{F},\mathcal{G})$. Thus, $ M\lessapprox\mu\nu{\lambda_1}^{-1}{\lambda_2}^{-1}.$
     We take a log-convex combination of this and $M\leq\frac{\sigma^2}{\delta^2}$ to get $$M\lessapprox\sigma^{1/2}\delta^{-1/2}\mu^{3/4}\nu^{3/4}\lambda_1^{-3/4}\lambda_2^{-3/4}. $$
     Moreover, by construction we have $\#{\mathcal{R}}\approx_\delta M\#\tilde{\mathcal{R}}_3$. From Lemma \ref{le:incid} (for $t=\rho, \eta=\sigma$ and $\mu,\nu$) we have
    \begin{equation}\label{eq:wolff}
        \#\tilde{\mathcal{R}}_3\lesssim\big(\#\mathcal{F}\#\mathcal{G}\big)^\varepsilon\Bigg[\Big(\frac{\#\mathcal{F}\#\mathcal{G}}{\mu\nu}\Big)^{3/4}+\frac{\#\mathcal{F}}{\mu}+\frac{\#\mathcal{G}}{\nu}\Bigg].
    \end{equation} 
    At this point, we may assume without loss of generality that
    \begin{equation*}
        \frac{\#\cG}{\nu}\leq\frac{\#\cF}{\mu}.
    \end{equation*}
    Using this assumption and multiplying \eqref{eq:wolff} by $\mu^{3/4}\nu^{3/4}$, we get
    \begin{gather}\label{eq:inc3}                \lambda_1^{3/4}\lambda_2^{3/4}\#\mathcal{R}\lessapprox\lambda_1^{3/4}\lambda_2^{3/4}M\#\tilde{\mathcal{R}}_3\lessapprox\sigma^{1/2}\delta^{-1/2}\Big(\#\mathcal{F}^{3/4}\#\mathcal{G}^{3/4}+\mu^{-1/4}\nu^{3/4}\#\mathcal{F}\Big).
    \end{gather}
    This would be sufficient for proving Lemma \ref{le:aux}, if we knew that \begin{gather} \label{eq:pencil} \mu^{-1/4}\nu^{3/4}\#\mathcal{F}\lesssim_\varepsilon\delta^{-\varepsilon}\big[\#\mathcal{F}^{3/4}\#\mathcal{G}^{3/4}+\#\mathcal{F}\big].
    \end{gather}
    However, this is a non-trivial estimate and can in fact fail if one does not assume that $\mathcal{F}$ and $\mathcal{G}$ are $(\delta,\alpha)$-broad for $\alpha\geq1/4$, as we shall show in Section \ref{sec:4} (Proposition \ref{thm:clam}). The rest of the proof is dedicated to showing (\ref{eq:pencil}). This is where the broadness hypothesis of Theorem \ref{thm:kakeya} will be used. \\\\
    We claim that if we have either
    \begin{equation}\label{eq:cond}
        \frac{\#\cF}{\mu}\lesssim_\varepsilon\delta^{-\varepsilon}\Big(\frac{\#\cG}{\nu}\Big)^3\ \mbox{ or }\ \nu\lesssim_\varepsilon\delta^{-\varepsilon}\mu^{1/3},
    \end{equation}
     then \eqref{eq:pencil} holds. Indeed,\\\\
     \textbf{i)} If $\frac{\#\cF}{\mu}\lesssim_\varepsilon\delta^{-\varepsilon}\Big(\frac{\#\cG}{\nu}\Big)^3$ holds, then we rearrange to get 
     $$\mu^{-1/4}\nu^{3/4}\#\cF\lesssim\delta^{-\varepsilon}\#\cF^{3/4}\#\cG^{3/4},$$
     which in turn implies \eqref{eq:pencil}.\\\\
     \textbf{ii)} If $ \nu\lesssim\delta^{-\varepsilon}\mu^{1/3}$ holds, then we have
     $$\mu^{-1/4}\nu^{3/4}\#\cF\lesssim\delta^{-\varepsilon}\#\cF,$$
     which also implies \eqref{eq:pencil}.\\\\
     We therefore need to show that one of the two inequalities in \eqref{eq:cond} always holds.
     If $\#\cG\lesssim t^{-1/2}$, then the $(\delta,\frac{1}{2})$-broadness hypothesis for $\#\cG$ implies $$\nu\lesssim1\lesssim\delta^{-\varepsilon}\mu^{1/3}, $$ and in this case we are done. Therefore, it suffices to prove \eqref{eq:cond} under the additional assumption that $\#\cG\gtrsim t^{-1/2}$. In this case, the $(\delta,\frac{1}{2})$-broadness hypothesis on $\cG$ yields
     $$\nu\lesssim t^{1/2}\#\cG.$$
     We end this subsection with the following simple observation:
     \begin{gather*}
        \frac{\#\mathcal{F}}{\mu}\leq\frac{1}{\mu}\sum_{\tilde{R}\in\tilde{\cR}}\{f\in\cF:\ f\sim\tilde{R}\}\lesssim\frac{1}{\mu}\mu\#\tilde{\cR}\approx\#\tilde{\mathcal{R}}_3,
    \end{gather*}
     so \eqref{eq:cond} would follow from
    \begin{gather}\label{eq:final3} 
        \#\tilde{\mathcal{R}}_3\lesssim C_\varepsilon\delta^{-\varepsilon}\Big(\frac{\#\mathcal{G}}{\nu}\Big)^3.
    \end{gather}
    In the next subsection, we will prove \eqref{eq:final3} (under the assumption that $\#\cG\gtrsim t^{-1/2}$) and thus conclude the proof of Theorem \ref{thm:kakeya}.
    
    \subsection{The Broadness Argument}\label{subsec:broadness}
    It is important to note here that the rectangles $\tilde{R}\in\tilde{\mathcal{R}}_3$ are adapted to tangent pairs $(f,g)\in\mathcal{F}\times\mathcal{G}$. However, it is still possible for many curves $g\in\mathcal{G}$ to be tangent to each other at longer lengths. This potential mismatch between tangency lengths is a phenomenon unique to the bilinear setting; one which was not present in previous works that involve tangency-pair-counting arguments withing a \textit{single} family of curves, such as \cite{Wolff}\cite{Zahl},\cite{Schlag}. In order to deal with this obstacle, we need to identify a \textit{broad} scale for the tangencies between curves within the same family. The following lemma is a ``two-ends reduction" similar in spirit to Step $2$ in the proof of Theorem $1.11'$ in \cite{Zahl2}. It allows us to find the largest $L\in[\sqrt{\sigma/\rho},1]$ such that most $g\in\cG$ are tangent to a rectangle of thickness $\sigma$ and length $L$.
    \begin{lemma}\label{le:robust}
        Fix a constant $B\in[1, \rho^{-\varepsilon}]$. Let $\tilde{\mathcal{R}}=h^\sigma(J)\in\tilde{\mathcal{R}}_3$, where $J$ is an interval of length $\sqrt{\sigma/\rho}$ and midpoint $\theta$. For $r\in[\sigma,\rho]$, let $J_r$ denote the interval centred at $\theta$ with length $\sqrt{\sigma/r}$ and let $S_{\tilde{R}}^r=h^\sigma(J_r)$ be the ``$r$-lengthening" of $\tilde{R}$. Then there exist $t_{\tilde{R}}\in[\sigma,\rho]$ and a $(\sigma,t_{\tilde{R}})$-rectangle $S_{\tilde{R}}$ that contains $\tilde{R}$ (up to dilation by a factor of $10$) such that 
        \begin{gather}\label{eq:S broad}
            \#\cG( S_{\tilde{R}})\geq Bt_{\tilde{R}}^\varepsilon\#\cG(\tilde{R})\sim t^\varepsilon_{\tilde{R}}\nu,
        \end{gather}
        and for every $r< t_{\tilde{R}}$ and every $(\sigma,r)$-rectangle $S$ that contains $\tilde{R}$ (up to dilation by a factor of $10$) we have
        \begin{gather*}
            \#\cG(S)\leq Br^\varepsilon\#\cG(\tilde{R}).
        \end{gather*}
    \end{lemma}
    \begin{proof}
         If we take $S=\tilde{R}$, then we have
        $$\#\cG(S) =\#\cG(\tilde{R})\geq B\rho^\varepsilon\#\cG(\tilde{R}),$$
        by choice of $B$. Thus, the set of $r\in[\sigma,\rho]$ such that there exists a $(\sigma,r)$-rectangle $S$ which contains $R$ and satisfies 
        \begin{equation}\label{eq:ebroad}
            \#\cG(S) \geq B r^\varepsilon\#\mathcal{G}(\tilde{R}).
        \end{equation}
        is non-empty. Let $t_{\tilde{R}}$ be the infimum of that set (over dyadic values) and let $S_{\tilde{R}}$ be a $(\sigma,t_{\tilde{R}})$-rectangle with the property
        $$\#\cG(S_{\tilde{R}})\geq B t_{\tilde{R}}^\varepsilon\#\cG(\tilde{R}).$$
        By definition, for every (dyadic) $r\in[\sigma,t_{\tilde{R}})$ and every $(\sigma,r)$-rectangle that contains $\tilde{R}$ we have
        \begin{equation}\label{eq:zahlbroad}
            \#\cG(S) \leq B r^\varepsilon\#\mathcal{G}(\tilde{R}).
        \end{equation}
        \end{proof}
        By dyadic pigeonholing, we may find $t\in[\sigma,1]$ and $\tilde{\mathcal{R}}_4\subset\tilde{\mathcal{R}}_3$ such that 
        $\#\tilde{\mathcal{R}}_4\approx\#\tilde{\mathcal{R}}_3$ and for each $\tilde{R}\in\tilde{\mathcal{R}}_4$ we have $t_{\tilde{R}}\in[t,2t]$. Let
        $$\mathcal{S}_0:=\{S_{\tilde{R}}: \tilde{R}\in\tilde{\mathcal{R}}_4\}.$$
        By a standard Vitali-type covering lemma, we may choose a maximal subfamily $\mathcal{S}\subset\mathcal{S}_0$ that consists of pairwise incomparable $(\sigma,t)$-rectangles and such that each $\tilde{R}\in\tilde{\cR}_4$ is contained in the 10-fold dilate of some $S\in\cS$.\\\\
        We now do a broad/narrow analysis. Let $K\geq 1$ be a large constant to be chosen later. We say that a pair $(g_1,g_2)\in\cG$ is $K$-transverse if 
        $$\frac{\sigma t}{K}\leq\Delta(g_1,g_2)\leq\sigma t.$$
        We say that $S\in\cS$ is $K$-broad if at least half of the pairs $(g_1,g_2)\sim S$ are $K$-transverse. We say that $S $ is $K$-narrow if it is not $K$-broad. If $(g_1,g_2)$ is $K$-transverse then by Lemma \ref{le: rect}, $g_1^\sigma\cap g_2^\sigma$ is the union of at most two rectangles of dimensions $\sigma\times\frac{\sigma}{\sqrt{\Delta(g_1,g_2)\rho}}\sim\sigma\times\sqrt{\frac{\sigma}{t\rho/K}}$. Thus, by comparing lengths, we get
        \begin{gather*}
            \sum_{S\sim(g_1,g_2)}1\lesssim \sqrt{\frac{K}{\rho}}.
        \end{gather*}
        Hence, we have
        \begin{gather}\label{eq:Sbroad}
            t^{2\varepsilon}\nu^2\#\cS_{broad} \lesssim
            \sum_{S \text{ broad}}\sum_{(g_1,g_2)\in\cG^2,\ (g_1,g_2)\sim S}1
            \lesssim\sum_{(g_1,g_2)\ K\text{-transverse}}\sum_{S\sim (g_1,g_2)}1\lesssim\#\cG^2\sqrt{\frac{K}{\rho}}.
        \end{gather}
        Now suppose $S=h^\sigma(J)\in\cS$ be narrow, where $S=S_{\tilde{R}}$ for some  $\tilde{R}\in\tilde{\cR}_4$. Then for at least half of the pairs $(g_1,g_2)\in\cG(S)^2$ we have $\Delta(g_1,g_2)\leq\frac{\sigma t}{K}\leq\frac{\sigma t_{\tilde{R}} }{K}$. Hence,
        \begin{gather*}
            \#\cG(S)^2/2\leq\sum_{\Delta(g_1,g_2)\leq\frac{\sigma t_{\tilde{R}}}{K}}1=\sum_{g_1\in\cG(S)}\sum_{g_2:\ \Delta(g_1,g_2)\leq\frac{\sigma t_{\tilde{R}}}{K}}1.
        \end{gather*}
        By pigeonholing, there exists a $g_*\in\cG(S)$ such that 
        $$\#\{g\in\cG(S):\ \Delta(g,g_*)\leq\frac{\sigma t_{\tilde{R}}}{K}\}\geq\#\cG(S)/2.$$
        Let $J^*$ be the interval concentric to $J$ with length $\sqrt{\frac{\sigma}{t_{\tilde{R}} /K}}$ and let $S^*:=g_*^\sigma(J^*)$.
        $S^*$ is a $(\sigma,\frac{t_{\tilde{R}} }{K})$-rectangle and $10{S}^*\supset S\supset\tilde{R}$. If $\Delta(g,g_*)\leq\frac{\sigma t_{\tilde{R}}  }{K}$, then as in the proof of Lemma \ref{le:tangency}, we may conclude that $g\sim S^*$. Therefore, 
        \begin{gather*}
            \#\cG(S^*)\geq\#\{g\in\cG(S):\ \Delta(g,g_*)\leq\frac{\sigma t_{\tilde{R}} }{K}\} \geq\frac{\#\cG(S)}{2}\geq\frac{B}{2} t_{\tilde{R}}^{\varepsilon} \#\cG(\tilde{R}).
        \end{gather*}
        Thus, by choosing $K=2^{\lceil 1/\varepsilon\rceil}$ we get
        \begin{gather}\label{eq:broadness}
            \#\cG(S^*)\geq \frac{B}{2} t_{\tilde{R}}^\varepsilon \#\cG(\tilde{R})\geq B\left(\frac{t_{\tilde{R}} }{K}\right)^\varepsilon\#\cG(\tilde{R}).
        \end{gather}
        This is a contradiction, as ${S}^*$ is a $(\sigma,\frac{t_{\tilde{R}} }{K})$-rectangle that contains $\tilde{R}$ (up to dilation by a factor of $10$) and $t_{\tilde{R}} $ is the minimal (dyadic) number such that property \eqref{eq:broadness} holds for $\tilde{R}$. Therefore, every $S\in\cS$ is broad and we have
        \begin{gather}
            \#\cS\lesssim t^{-2\varepsilon}\sqrt{\frac{K}{\rho}}\left(\frac{\#\cG}{\nu}\right)^2.
        \end{gather}
         By comparing the lengths of $S\in\cS$ and $\tilde{R}\in\tilde{\cR}_4$, we get
        \begin{equation}\label{eq:length2}
            \#\{\tilde{R}\in\tilde{\mathcal{R}}_4:\ \tilde{R}\subset 10S\}\lesssim\sqrt{\frac{\rho}{t}}.
        \end{equation}
        We combine \eqref{eq:Sbroad} and \eqref{eq:length2} to obtain
        \begin{gather}\label{eq:case3}
            \nonumber \#\tilde{\mathcal{R}}_4=\sum_{S\in\mathcal{S}}\#\{\tilde{R}\in\tilde{\mathcal{R}}_4:\ \tilde{R}\subset 10 S\}\lesssim\sqrt{\frac{\rho}{t}}\#\cS\\
           \lesssim t^{-2\varepsilon}K^{1/2}t^{-1/2}\left(\frac{\#\cG}{\nu}\right)^2.
        \end{gather} 
        As we remarked at the end of Subsection \ref{subsec:coarse}, under the assumption that $\#\cG\gtrsim t^{-1/2}$ we have  
        \begin{gather*}
            t^\varepsilon \nu\lesssim\#\cG(S)\lesssim t^{1/2}\#\cG,
        \end{gather*}
        from the $(\delta,\frac{1}{2})$-broadness hypothesis on $\cG$. By rearranging, we get
        \begin{gather*}
            t^{-1/2}\lesssim t^{-\varepsilon}\frac{\#\cG}{\nu}
        \end{gather*}
        Therefore, \eqref{eq:case3} yields
        \begin{gather*}
            \#\tilde{\cR}_3\lesssim 2^{1/2\varepsilon}\sigma^{-3\varepsilon}\left(\frac{\#\cG}{\nu}\right)^3.
        \end{gather*}
        
    \section{On the Sharpness of the Main Estimates}\label{sec:4}
    \subsection{Sharpness of Planar Curved Kakeya Estimate}
    We begin by showing that the powers of $\delta$ and $\rho$, as well as the $L^p$ exponent $3/4$ in the conclusion (\ref{eq:parabola}) of Theorem \ref{thm:kakeya} cannot be improved (except for potentially an $\varepsilon$-improvement on the power of $\delta$).\\\\
    To see that the powers of $\delta,\rho$ are optimal, let $f(s)=\frac{\rho}{2}s^2,\ g(s)=-\frac{\rho}{2}s^2$ and $\mathcal{F}=\{f\},\mathcal{G}=\{g\}$. Then we have
    \begin{equation*}
        |f(s)-g(s)|\leq\delta\iff\rho |s|^2\leq\delta\iff |s|\leq\sqrt{\delta/\rho},
    \end{equation*}
    and therefore $f^\delta\cap g^\delta$ is a curvilinear rectangle of dimensions $\sim\delta\times\sqrt{\frac{\delta}{\rho}}.$ Thus, the left-hand side of \eqref{eq:parabola} is $\sim\delta^{3/2}\rho^{-1/2}$, which is evidently equal to the right-hand side of \eqref{eq:parabola}.
    \begin{proposition}
        The smallest $L^p$ exponent such that the conclusion \eqref{eq:parabola} of Theorem \ref{thm:kakeya} holds is $3/4$.
    \end{proposition}
    \begin{proof}
        Let $B_1,\ B_2$ be $\rho$-separated $\tau$-balls of radius $\rho$, and let  $\mathcal{F},\mathcal{G}$ be maximal $\delta$-separated sets contained in $B_1$ and $B_2$ respectively. By construction, $(\mathcal{F},\mathcal{G})$ is a $\rho$-bipartite pair and $\#\mathcal{F}\sim\#\mathcal{G}\sim\Big(\frac{\rho}{\delta}\Big)^3$. For each $p=(s,y)\in[0,1]^2$ and $f(s)=as^2+bs+c$ we have
        \begin{gather*}
            f\in\mathcal{F}(p)\iff|as^2+bs+c-y|\lesssim\delta.
        \end{gather*}
        The above condition defines a slab of thickness $\delta$ and radius $\rho$ in $(a,b,c)$-space. Thus, $m_\mathcal{F}(p)=\sum_{\mathcal{F}}\chi_{f^\delta}(p)\sim(\rho/\delta)^2$ similarly $m_\mathcal{G}(p)\sim(\rho/\delta)^2$, so that 
        $$ \int_{[0,1]^2} \Big( \sum_{f \in \mathcal{F}} \chi_{f^\delta} \Big)^{p} \Big( \sum_{g \in \mathcal{G}} \chi_{g^\delta} \Big)^{p}\sim \delta^{-4p}. $$
        On the other hand, we have
        $$\delta^{3/2}\Big(\#\mathcal{F}^p\#\mathcal{G}^p+\#\mathcal{F}+\#\mathcal{G}\Big)\sim\delta^{3/2}\#\mathcal{F}^p\#\mathcal{G}^p\sim\delta^{3/2-6p}. $$
        Thus, for  \eqref{eq:parabola} to hold, we need $\delta^{-4p}\leq\delta^{3/2-6p},$ which yields $p\geq3/4$.
    \end{proof}
    Next, we investigate what is the minimal exponent $\alpha$ such Theorem \ref{thm:kakeya} holds for $(\delta,\alpha)$-broad families of quadratics.
    \begin{proposition}\label{thm:clam}
    If Theorem \ref{thm:kakeya} holds for some $\alpha>0$, then necessarily $\alpha\geq 1/4$.
    \end{proposition}
    \begin{proof}
        Let $0<\delta\leq t\leq 1$ and let $\mu\sim t/\delta, \nu\leq\delta^{-1},\mu\leq N\leq\delta^{-1}$ be non-negative integers. Let $\{R_j\}_{j=1}^{N/\mu}$ be a family of pairwise disjoint $(\delta,t)$-rectangles. For each $j\in[1,\frac{N}{\mu}]$ let $\{f_{j,k}\}_{k\in[1,\mu]}$ be a family of quadratic polynomials that are all tangent to $R_j$ and let $\mathcal{F}:=\{f_{j,k}: j,k\}.$ Divide each $R_j$ length into $(\delta,1)$-rectangles $R_{j,l},\ 1\leq l\leq\sqrt{\frac{\delta}{t}}\Big/\sqrt{\delta}=t^{-1/2}$ and for each $j,l$ let $g_{j,l,n},\ 1\leq n\leq\nu$ be  quadratic polynomials which are tangent only to $R_{j,l}$ and such that $\tau(f_{j,k},g_{j,l,n})\geq 1/2$. Define $\mathcal{R}:=\{R_{j,l}\},\ \mathcal{G}=\{g_{j,l,n}\}$.  We have $$\#\mathcal{F}=\mu\frac{N}{\mu}=N,\ \#\mathcal{R}=t^{-1/2}\frac{N}{\mu}=\delta^{-1/2}\frac{N}{\mu^{3/2}},\ \#\mathcal{G}\sim\#\mathcal{R}\nu.$$
    If $f\in\mathcal{F}$ is tangent to some $R_j$, then since $R_{j,l}\subset R_j$, we also have that $f\sim R_{j,l}$. Thus, $$\#\{f\in\mathcal{F}:\ f\sim R_{j,l}\}=\#\{f\in\mathcal{F}:\ f\sim R_{j}\}=\mu,$$
    and
    $$\#\{g\in\mathcal{G}:\ g\sim R_{j,l}\}\sim\nu.$$
    Suppose Theorem \ref{thm:kakeya} holds for all families $\cF,\cG$ that are $(\delta,\alpha)$-broad for some $\alpha>0$. Since by construction the rectangles $R_{j,l}$ are pairwise disjoint and  account for all the tangencies between pairs $(f,g)\in\cF\times\cG$, the left-hand side of \eqref{eq:parabola} is equal to
    \begin{equation*}
        \delta^{3/2}t^{-1/2}\#\cR\mu^{3/4}\nu^{3/4},
    \end{equation*}
    whereas the right-hand side is equal to
    \begin{equation*}
        \delta^{3/2}t^{-1/2}\#\cF^{3/4}\#\cG^{3/4}.
    \end{equation*}
    Therefore, we must have
    \begin{gather}
        \mu^{3/4}\nu^{3/4}\#\mathcal{R}\lesssim\#\mathcal{F}^{3/4}\#\mathcal{G}^{3/4}\sim\#\mathcal{F}^{3/4}\#\mathcal{R}^{3/4}\nu^{3/4},\nonumber \\
        \iff \mu^{3/4}\#\mathcal{R}^{1/4}\lesssim\#\mathcal{F}^{3/4} \nonumber\\
        \iff \mu^3\delta^{-1/2}\frac{N}{\mu^{3/2}}\lesssim N^3 \nonumber\\
        \iff \mu\lesssim\delta^{1/3}N^{4/3}. \label{eq:mu1}
    \end{gather}
    We now apply (\ref{eq:broad}) to $\mathcal{F}$ and $R_{j}$ to get
    \begin{gather}
        \mu\lesssim t^\alpha N=\delta^\alpha\mu^\alpha N \nonumber\\
        \iff \mu\lesssim\delta^{\alpha/(1-\alpha)}N^{1/(1-\alpha)}.\label{eq:mu2}
    \end{gather}
    By comparing \eqref{eq:mu1} and \eqref{eq:mu2}, we conclude that $\alpha$ must be at least $1/4$. Moreover, since \eqref{eq:mu1} is equivalent to \eqref{eq:pencil}, it is also the case that \eqref{eq:pencil} fails unless $\alpha\geq 1/4$ (at least in this particular example).
\end{proof}
\begin{remark}
    It is an interesting question whether a $(\delta,\alpha)$-broadness hypothesis on $\mathcal{F}$ and $\mathcal{G}$ with $\alpha\in[\frac{1}{4},\frac{1}{2}]$ is sufficient to prove Theorem \ref{thm:kakeya}. The exponent $1/2$ seems to be the smallest exponent compatible with the argument presented here.
\end{remark}
    \subsection{Sharpness of the Bilinear Tube Estimate}
    Interestingly, we do not know if the exponent $3/4$ in Theorem \ref{thm:main} is sharp. The best bound we can obtain is the following.
    \begin{proposition}
        Suppose that the conclusion \eqref{eq:main} of Theorem \ref{thm:main} holds for some exponent $p>0$ in place of $3/4$. Then $p\geq 2/3$.
    \end{proposition}
    \begin{proof}
    Let $P\subseteq[0,1]^2$ be a parabolic $\delta$-net. Define $\cT_j=\{T_{e_j}^\delta(p):\ p\in P\}$. Then $\#\cT_1=\#\cT_2\sim\delta^{-3}$. Moreover, for each $j=1,2$ the tubes $T_j\in\cT_j$ are essentially disjoint and folliate $B_d(0,1)$. Thus,
    \begin{equation*}
        \int\Big(\sum_{T_1\in\cT_1}\chi_{T_1}\Big)^p\Big(\sum_{T_2\in\cT_2}\chi_{T_2}\Big)^p\sim 1.
    \end{equation*}
    On the other hand,
    \begin{equation*}
        \delta^4(\#\cT_1)^{p}(\#\cT_2)^p\sim\delta^{4-6p}.
    \end{equation*}
    Assuming that \eqref{eq:main} holds with exponent $p$, we may compare the two terms to get $p\geq 2/3$.
    \end{proof}
    It is an open question whether the $L^p$ exponent $3/4$ in Lemma  \ref{le:zahl} is sharp. In \cite{Wolff}, Wolff claimed that the exponent $3/4$ is the best that can be obtained with the techniques used in that paper (which rely on results in computational geometry, see \cite{Wolff} for further details). To the author's knowledge, there has not been any improvement explicitly on Wolff's original result. Recently, Maldague and Ortiz proved a better incidence-geometric bound for well-spaced circles in \cite{MO}. While their result is not directly applicable to parabolas, or in the bilinear setting, it suggests that improvements on the exponent $3/4$ might be possible using a different approach. We close this section with a conjecture:
    \begin{conjecture}\label{conj:sharp}
        The sharp $L^p$ exponent in Theorem \ref{thm:main} is $2/3$.
    \end{conjecture}
    Conjecture \ref{conj:sharp} is also supported by the following consideration: If all tubes in $\cT_j$ are chosen to have horizontal direction $e_j$ for $j=1,2$ then the statement of Theorem \ref{thm:main} can be viewed as a ``thickened" version of the Szemer\'edi--Trotter theorem. This can be seen using point-line duality, in the way that the authors of \cite{FOP} did. In fact, if we assume that the tubes in $T_j\in\cT_j$ are pairwise disjoint and that their cores belong in a set of horizontal lines $\cL_j$, then we have
    \begin{gather}\label{eq:ST}
        \int\Big(\sum_{T_1\in\cT_1}\chi_{T_1}\Big)^{2/3}\Big(\sum_{T_2\in\cT_2}\chi_{T_2}\Big)^{2/3}\sim\delta^4\#\cI (\cL_1,\cL_2),
    \end{gather}
    where $\cI(\cL_1,\cL_2)$ denotes the set of incidences between $\cL_1$ and $\cL_2$. By point-line duality and the Szemer\'edi--Trotter theorem, the right-hand side of \eqref{eq:ST} is
    \begin{equation*}
        \lesssim\delta^4\Big(\cT_1^{2/3}\#\cT_2^{2/3}+\#\cT_1+\#\cT_2\Big).
    \end{equation*}
    One can also apply Theorem $1.2$ of \cite{FOP} to bound the number of approximate $\delta$-incidences between $\cL_1$ and $\cL_2$ by $\delta^{-1/3}\#\cL_1^{2/3}\#\cL_2^{2/3}$, which leads to the following bound for \eqref{eq:ST}:
    \begin{equation*}
        \delta^{4-1/3}\#\cT_1^{2/3}\#\cT_2^{2/3}.
    \end{equation*}
    We refer the interested reader to \cite{FOP} for more details on point-line duality in the context of horizontal lines in the Heisenberg group.
\subsection{Concluding Remarks}
It is also natural to ask if these results can be extended to higher dimensions. To the author's knowledge, there are no results about the Heisenberg Kakeya conjecture in higher dimensions. Incidence estimates for curves in dimension $3$ and higher are quite limited and in general seem to be much harder than those in two dimensions. The techniques presented here are essentially limited to two dimensions.

\printbibliography
\end{document}

%% file: preamble.tex
\usepackage[utf8]{inputenc}
\usepackage{bbm}
\usepackage{pdfpages}
\usepackage{tcolorbox}
\usepackage{esint} % for \fint
\usepackage{amsmath,amssymb,amsthm}

\theoremstyle{plain}

\newtheorem{proposition}{Proposition}

\newtheorem{conjecture}{Conjecture}
\newtheorem{example}{Example}

\theoremstyle{definition}

\newtheorem*{first key estimate}{First key estimate}
\newtheorem*{second key estimate}{Second key estimate}

\newcommand{\bbH}{\mathbb{H}}

\newcommand{\bbL}{\mathbb{L}}

\newcommand{\bbS}{\mathbb{S}}
\newcommand{\R}{\mathbb{R}}

\newcommand{\bbW}{\mathbb{W}}

\newcommand{\Z}{\mathbb{Z}}

\newcommand{\cF}{\mathcal{F}}
\newcommand{\cG}{\mathcal{G}}
\newcommand{\cH}{\mathcal{H}}
\newcommand{\cI}{\mathcal{I}}

\newcommand{\cL}{\mathcal{L}}

\newcommand{\cQ}{\mathcal{Q}}
\newcommand{\cR}{\mathcal{R}}
\newcommand{\cS}{\mathcal{S}}
\newcommand{\cT}{\mathcal{T}}

\newcommand{\dist}{\mathrm{dist}}